\documentclass[12pt,english,russian]{article}
\usepackage[T2A]{fontenc}
\usepackage[utf8]{inputenc}
\usepackage[dvipsnames]{xcolor}
\usepackage{amsfonts,amsmath,amssymb,mathtext,amsthm}
\usepackage{cite,enumerate,float,indentfirst}
\usepackage{graphics,epsfig,graphicx,epstopdf,amscd,latexsym,verbatim,hyperref,ulem}
\usepackage[top = 30mm, bottom = 20mm, left = 30mm, right = 25mm]{geometry}
\usepackage{centernot}

\usepackage{pgfplots,comment}
\pgfplotsset{compat=1.9}
\usepackage{filecontents}
\usetikzlibrary{arrows.meta}
\usepackage{tikz-cd}
\usepackage{blindtext, rotating}

\newtheorem{lemma}{Lemma}[section]
\newtheorem{example}{Example}[section]
\newtheorem{theorem}{Theorem}[section]

\newtheorem{definition}{Definition}[section]
\newtheorem{remark}{Remark}[section]
\newtheorem{problem}{Problem}[section]
\newtheorem{question}{Question}[section]

\def\Sh{\mathrm{Sh}\,}

\def\Imm{\mathrm{Im}\,}

\def\Span{\mathrm{Span}}
\def\gcd{\mathrm{gcd}}
\def\charr{\mathrm{char}\,}

\newcounter{pict}
\begin{document}

\begin{center}
{\Large Monomial Rota---Baxter operators of weight zero and averaging operators on the polynomial algebra}

\smallskip

Artem Khodzitskii \\
Novosibirsk State University
\end{center}

\begin{abstract}
Starting with the work S.H.~Zheng, L.~Guo and M.~Rosenkranz (2015),
Rota---Baxter operators are studied on the polynomial algebra.
Injective Rota---Baxter operators of weight zero on $F[x]$ were described in 2021.
We classify the following classes of monomial Rota---Baxter operators of weight zero on the polynomial algebra $F[x,y]$ and its augmentation ideal $F_0[x,y]$: 
1) non-increasing in degree that do not contain monomials in the kernel,
2) mapping all monomials to themselves with a~coefficient.
In the context of these sets of operators, we show how one may 
define a monomial averaging operator by a given RB-operator and vice versa.


\end{abstract}

\section{Introduction}
Rota---Baxter operators are an algebraic analogue of the integral operator.
They are widely used in various areas of mathematics and physics,
including the study of Yang---Baxter equations, 
combinatorics, and quantum field theory~\cite{GuoMonograph}.
Rota---Baxter operators were first defined by G.~Baxter in 1960~\cite{Baxter}.
However, the relation defining Rota---Baxter operators 
had been already appeared in the works of 
F.G.~Tricomi~\cite{Tricomi} and M.~Cotlar~\cite{Cotlar}.

\begin{definition}
A linear operator $R$ on an algebra $A$ defined over a field $F$ 
is called Rota---Baxter operator (RB-operator, for short)
if the following relation
\begin{gather}\label{RBO}
R(a)R(b) = R\big(R(a)b + a R(b) + \lambda a b\big)
\end{gather}
holds for all $a, b \in A$.
Here $\lambda\in F$ is a fixed scalar called a weight of $R$.
\end{definition}

When $\lambda = 0$, the identity~\eqref{RBO} is a generalization of the integration by parts formula.
The most active in the study of Rota---Baxter operators in 1960--1990s was G.-C.~Rota~\cite{Rota}.
In the 1980s, Rota---Baxter operators were rediscovered in the context 
of the classical and modified Yang---Baxter equations
from mathematical physics~\cite{BelaDrin82,Semenov83}.
Recently, Rota---Baxter operators have been defined on different algebraic structures: on groups~\cite{Guo2020}, Hopf algebras~\cite{Goncharov2020}, and coalgebras~\cite{Coalgebras}.

Along with Rota---Baxter operators,
averaging operators are also of particular interest.

\begin{definition}
A linear operator $T$ on an algebra $A$ defined over a field $F$
is called averaging operator if the following relations hold for all $a,b \in A$:
\begin{gather}\label{Averaging}
T(a)T(b) = T(T(a)b) = T(a T(b)).
\end{gather}
\end{definition}

Averaging operators were first introduced by O.~Reynolds~\cite{Reynolds} 
in the context of turbulence theory.
At the beginning of the 20th century, J. Kamp\'{e} de F\'{e}riet, who noted the importance of  averaging operators, studied them along with Reynolds operators for over 30 years~\cite{KampeDeFeriet}.
Averaging operators have also been studied in functional analysis~\cite{RotaAvOp};
a~connection between averaging operators and the conditional mean was found in~\cite{Moy}. 
In~\cite{Bong}, N.H.~Bong studied operators that are simultaneously 
Rota---Baxter and averaging operators.
More information about averaging operators can be found in~\cite{PeiGuoAveraging}.

In this paper, we study the connection between averaging operators
and Rota---Baxter operators on a polynomial algebra.
We focus on the case of two variables.
We consider and get partial classification of monomial RB-operators of weight zero and 
averaging operators on the algebras $F[x,y]$ and~$F_0[x,y]$, 
where $F_0[Z]$ denotes the augmentation ideal in $F[Z]$.

The active study of RB-operators on the polynomial algebra was started
with the work of S.H.~Zheng, L.~Guo, and M.~Rosenkranz in 2015~\cite{Monom2}.
Injective monomial RB-operators of weight zero on $F[x]$ were described there.
The monomial condition of an operator means that each monomial is mapped to
some monomial with a~coefficient from the ground field~$F$.
In 2016, H.~Yu described monomial RB-operators of any weight~$\lambda$ on $F[x]$~\cite{Monom}.
In 2020, V.Y.~Gubarev classified monomial RB-operators on $F_0[x]$~\cite{MonomNonunital}.

In~\cite{GubPer}, the conjecture of S.H.~Zheng, L.~Guo, and M.~Rosenkranz
on the description of injective (not necessarily monomial)
RB-operators on $F[x]$, where $F$ has characteristic zero, was confirmed. 
In~\cite{ReprRB,ReprRB2,ReprRB3}, finite-dimensional irreducible representations 
of certain Rota---Baxter algebras on $F[x]$ were described.

Although the examples of RB-operators on $F[x,y]$ had already appeared~\cite{Ogievetsky,Viellard-Baron},
a~systematic study of monomial Rota---Baxter operators of nonzero weight on the algebra $F[x,y]$ began in~\cite{Khodzitskii}.
Their connection with monomial homomorphic averaging operators was noted.
In particular, the author described in~\cite{Khodzitskii} monomial Rota---Baxter operators of nonzero weight on $F[x,y]$, which come from monomial homomorphic averaging operators.
In the current work, we will see that a connection between monomial RB-operators of weight zero and averaging operators is much deeper.

The idea to study monomial linear operators on a polynomial algebra is quite natural.
For example, monomial derivations were used to construct counterexamples
to the Hilbert's 14th problem. Monomial derivations on $F[X]$ have been studied 
in a series of papers~\cite{NowZiel,Ollagnier,Kitazawa,Essen}.
Note that if an RB-operator is invertible, then its formal inverse is a~derivation.

The paper is organized as follows.

In \S2, we consider the basic properties of RB-operators and averaging operators.

In \S3, we suggest a~method for constructing an RB-operator of weight zero 
and an averaging operator on the polynomial algebra $F[X]$ ($F_0[X]$) by a given monomial operator~$S$ satisfying corresponding conditions (see Lemmas~\ref{RB_by_lin_op_lemma} and~\ref{Averaging_by_lin_op_lemma}).
Further, we state problems (Problem~\ref{MainProblem} and~\ref{MainProblem2})
about the classification of monomial RB-operators~$R$ having the form 
$$
R(z) = \alpha_z T(z), \,\, z \in M(X)\,(M_0(X)),\,\, \alpha_z \in F, 
$$
where $T$ is a fixed monomial averaging operator, and vice versa.
Such formulation allows us to reduce the problems to solving recurrence relations on $\alpha_z$.
The similar problem stated for monomial RB-operators of nonzero weight on $F[x,y]$ coming from homomorphic averaging operators was studied in~\cite{Khodzitskii}.
We post a question (Question~\ref{AdditionalProblem}): Do we obtain all Rota---Baxter operators or averaging operators in such a manner?
The next paragraphs are devoted to consideration of this question.

In \S4, we describe monomial averaging operators on $F[x]$ and $F_0[x]$ (Theorem~\ref{Averaging_on_F[x]_and_F0[x]}).
On $F_0[x]$, we indeed obtain all monomial RB-operators by averaging ones, and vice versa. On $F[x]$, this is not so, since the unit may lie in the image of an averaging but not an Rota---Baxter operator. However, this property is the only source of the difference between the set of all monomial averaging operators and the set of averaging operators coming from monomial RB-operators.

In \S5, we describe monomial non-increasing in degree RB-operators of weight zero and averaging operators on $F_0[x,y]$ that do not contain monomials in the kernel (Theorems~\ref{RB_not_incr_and_ker=0} and~\ref{Averaging_not_incr_and_ker=0}). 
We consider a~question analogous to~\ref{AdditionalProblem} and 
find that the described classes coincide up to the coefficients.

In \S6, we describe all monomial RB-operators of weight zero and averaging operators
on $F[x,y]$ and $F_0[x,y]$ that map each monomial to itself with some coefficient from~$F$ (Theorems~\ref{AvOp_Case3.1_theorem} and~\ref{AvOp_Case3.1_theorem_averaging}).
Such sets of operators have no significant differences on $F_0[x, y]$ 
and coincide up to the coefficients.
On $F[x, y]$, the monomial averaging operators are defined by the same formulas as on $F_0[x, y]$, 
taking into account the presence of a unit in algebra $F[x, y]$.
However, a~similar set of RB-operators contains only the zero operator.

\section{Preliminaries}

A~set of variables we denote by $X$,
where $X = \{ x_i \mid i \in I\}$ and $I \neq \emptyset$ is a set of indexes.
We also use the notation $F_0[X] = (X)$
for the ideal in $F[X]$ generated by $X$
and $F^* = F \setminus \{0\}$.

\begin{definition}
Given $X \neq \emptyset$.
The set of monomials of $F[X]$ is defined as
$$
M(X) = 
\bigg\{ 
\prod_{i \in I} x_i^{\beta_i} \mid 
x_{\alpha} \in X, \, \beta_i \geqslant 0
\bigg\}.
$$
There are only finitely many nonzero scalars $\beta_i$
in the product $\prod_{i \in I} x_i^{\beta_i}$.

For the algebra $F_0[X]$ we define the set $M_0(X)$ as a~subset of $M(X)$ 
of elements $\prod_{i \in I} x_i^{\beta_i}$
satisfying the condition $\sum\limits_{i \in I} \beta_i > 0$.
\end{definition}

\begin{definition}[\!\!\cite{Monom2}]
Given $X \neq \emptyset$.
A linear operator~$L$ on $F[X]$ ($F_0[X]$) is called monomial
if for any $t \in M(X)$ ($M_0(X)$) there exist
$z_t \in M(X)$ ($M_0(X)$) and $\alpha_t \in F$
such that $L(t) = \alpha_t z_t$.
\end{definition}

\begin{definition}
Let $X$ be a nonempty set and let $L$ be a monomial operator on $F[X]$.
Define the set
$$
\Sh L = 
\{ z \in M(X) \mid 
L(z) = \alpha w, \, \alpha \in F^*, \,
z \neq w \in M(X)\}.
$$
We define $\Sh L$ for the algebra $F_0[X]$ similarly.
\end{definition}

We give the following well-known property of the Rota---Baxter operators.

\begin{lemma}\label{ker(R)IsIm(R)Module}
Let $R$ be an RB-operator of weight zero on an algebra $A$.
Then $\Imm R\cdot\ker R$, 
$\ker R \cdot\Imm R \subseteq \ker R$.
\end{lemma} 

\begin{lemma} [\!\!\cite{Cao}] \label{closure_and_modules_lemma}
Let $T$ be an averaging operator on an algebra $A$.
Then the following statements hold:

a) $\Imm T$ is a subalgebra in $A$.

b) $\Imm T\cdot\ker T$, $\ker T\cdot\Imm T\subseteq \ker T$.

\noindent If additionally $A = F[X]$ ($F_0[X]$), 
where $X \neq \emptyset$ and $T$ is monomial, then

c) $\Imm T\setminus\{0\}\cdot \Sh T$, 
$\Sh T\cdot \Imm T\setminus\{0\} \subseteq \Sh T$.
\end{lemma}

\begin{proof}
a) and b) have been proven in~\cite{Cao}.

c) Let us consider a monomial operator~$T$.
Let $a \in \Sh T$ and let $b \in M(X)$ ($M_0(X)$) be such that $T(b)\neq 0$.
Since $0\neq T(a)T(b) = T(aT(b))$ and 
$aT(b) = \gamma T(a)T(b)$ does not hold for any $\gamma \in F$,
one derives that $aT(b) \in \Sh T$. We have proved c).
\end{proof}

An useful observation for the classification of operators is that 
both Rota---Baxter operators and averaging operators can be defined 
up to conjugation by an automorphism of the algebra 
and multiplication by a nonzero scalar.

\begin{lemma} [\!\!\cite{GuoMonograph,BGP}] \label{RB_under_automorphism}
Let $A$ be an algebra and let $R$ be an RB-operator of weight $\lambda$ on $A$.

a) The operator $\alpha^{-1} R$ for any $\alpha \in F^*$
is an RB-operator of weight $\alpha\lambda$ on~$A$.

b) The operator  $\psi^{-1} R \psi$ for any $\psi \in \mathrm{Aut}(A)$
is an RB-operator of weight $\lambda$ on~$A$.
\end{lemma}

\begin{lemma} \label{Averaging_under_automorphism}
Let $A$ be an algebra and let $T$ be an averaging operator on $A$.

a) The operator $\alpha^{-1} T$ for any $\alpha \in F^*$ 
is an averaging operator on~$A$.

b) The operator  $\psi^{-1} T \psi$ for any $\psi \in \mathrm{Aut}(A)$
is an averaging operator on~$A$.
\end{lemma}

\begin{proof}
The relations~\eqref{Averaging} obviously hold true, when they are multiplied by any nonzero scalar,
thus, a) follows.

b) Denote $P = \psi^{-1} T \psi$.
So we have
\begin{gather*}
P(a)P(b) =
(\psi^{-1} T \psi)(a)(\psi^{-1} T \psi)(b) =
\psi^{-1}[T(\psi(a))]\psi^{-1}[T(\psi(b))] = 
\psi^{-1}[T(\psi(a))T(\psi(b))], \\ 
P(P(a)b) = 
(\psi^{-1} T \psi) [(\psi^{-1} T \psi)(a) b] = 
\psi^{-1} [T(T(\psi(a))\psi(b))], \\
P(aP(b)) = 
(\psi^{-1} T \psi) [a (\psi^{-1} T \psi)(b)] = 
\psi^{-1} [T(\psi(a) T(\psi(b)))].
\end{gather*}
Therefore, $\psi (P(a)P(b)) = \psi(P(P(a)b)) = \psi(P(aP(b)))$
and~\eqref{Averaging} holds for $P$, since $\psi$~is an automorphism.
\end{proof}

\begin{lemma}[\!\!\cite{BGP}] \label{RB_R(1)_in_F}
Let $A$ be an unital algebra and let $R$ be an RB-operator of weight zero on $A$.
If $R(1) \in F1$, then $R(1) = 0$, $R^2 = 0$, and $\Imm R\subset \ker R$.
\end{lemma}

By Lemma~\ref{RB_R(1)_in_F}, there is no RB-operator $R$ of weight zero
on an unital algebra satisfying $1 \in \Imm R$.
On the contrary, there exists an averaging operator~$T$ such that $1 \in \Imm T$.

\begin{example}
Let $A$ be an unital algebra and let $T$ be an averaging operator on $A$
defined as follows: $T(a) = \varepsilon(a)$, 
where $\varepsilon \colon A \to F$ is a linear functional.
Then relations~\eqref{Averaging} hold for all $a, b \in A$, 
since $T(a)T(b) = \varepsilon_a \varepsilon_b$ and
$T(T(a)b) = T(\varepsilon_a b) = \varepsilon_a \varepsilon_b$.
\end{example}

In the case of averaging operators, a~result similar to Lemma~\ref{RB_R(1)_in_F} holds:

\begin{lemma}\label{Averaging_R(1)_in_F}
Let $A$ be an unital algebra and let $R$ be an averaging operator on $A$.
If $T(1) = \lambda \in F$, then $T^2(a) = \lambda T(a)$ for any $a \in A$.
In particular, when $\lambda = 0$, we have $T^2 = 0$ and $\Imm T\subset \ker T$.
\end{lemma}

\begin{proof}
The statement follows directly from~\eqref{Averaging}:
$$
\lambda T(a) = T(1)T(a) = T(1\cdot T(a)) = T^2(a), \quad a \in A.  \qedhere
$$
\end{proof}

\section{Connection between averaging operators and Rota---Baxter operators}

In this section, we explore the connection between
averaging operators and RB-operators
on the polynomial algebra.
We state the connection between RB-operators of weight zero
defined on $F[X]$ and RB-operators of weight zero defined on $F_0[X]$,
and also between averaging operators defined on $F[X]$ and on $F_0[X]$, respectively.
Besides this we show how to construct a monomial RB-operator by an averaging one and vice versa.

\begin{lemma}\label{RB_F0[X]<->F[X]}
Let $X$ be a nonempty set and let $F$ be a field of characteristic zero.

a) If $P$ is an RB-operator of weight zero on $F_0[X]$ such that $P^2 = 0$,
then the operator~$R$ defined as 
\begin{equation} \label{RB_F0[X]<->F[X]_formula}
R(t) = 
\begin{cases}
P(t), & \deg t > 0, \\
0, & \deg t = 0,
\end{cases}
\end{equation}
where $t \in M(X)$, is an RB-operator of weight zero on $F[X]$.

b) If $R$ is an RB-operator of weight zero on $F[X]$
and $F_0[X]$ is $R$-invariant,
then the operator $P =  R|_{F_0[X]}$ is an RB-operator of weight zero on~$F_0[X]$.
In particular, if $R(1) = 0$, then $P^2 = 0$.
\end{lemma}

\begin{proof}
We consider~\eqref{RBO} for defined operators in both cases.

a) If $a, b\in M_0(X)$, then in~\eqref{RBO} we obtain
$$
R(a)R(b) = P(a)P(b) = P(P(a)b + a P(b)) = R(R(a)b + a R(b)).
$$
If $a, b \in F1$, then~\eqref{RBO} is the trivial identity $0 = 0$.
Let $a\in M_0(X)$, $b \in F1$, then
$$
R(a)R(b) = 0,\quad 
R(R(a)b + aR(b)) = P(P(a)b) = bP^2(a) = 0,
$$
as required.

b) Since $F_0[X]$ is $R$-invariant, the relations
$$
R(a)R(b) = P(a)P(b) = P(P(a)b + a P(b)) = R(R(a)b + a R(b))
$$
hold for all $a, b \in M_0(X)$.
If $R(1) = 0$, then $R(1)R(1) = 0 = R(R(1)1 + 1R(1))$.
By Lemma~\ref{RB_R(1)_in_F}, we have $R^2 = 0$, hence $P^2 = 0$. 
It implies that relation~\eqref{RBO} holds, when $a\in F1$ or $b\in F1$.
\end{proof}

\begin{remark}
Let $R$ be a monomial RB-operator of weight zero on $F[X]$, 
then subspace $F_0[X]$ is $R$-invariant.
\end{remark}

We have a similar connection
between averaging operators on $F[X]$ and $F_0[X]$.

\begin{lemma}\label{Averaging_F0[X]<->F[X]}
Let $X$ be a nonempty set.

a) If $P$ is an averaging operator on $F_0[X]$
such that $P^2 = \lambda P$, $\lambda \in F$,
then the operator~$T$ defined by the rule
\begin{equation} \label{Averaging_F0[X]<->F[X]_formula}
T(z) = 
\begin{cases}
P(z), & \deg z > 0, \\
\lambda z, & \deg z = 0,
\end{cases}
\end{equation}
where $z \in M(X)$, is an averaging operator on $F[X]$.

b) If $T$ is an averaging operator on $F[X]$
and $F_0[X]$ is $T$-invariant,
then the operator $P =  T|_{F_0[X]}$ is an averaging operator on~$F_0[X]$. 
In particular, if $T(1)\in F1$, then $P^2 = \lambda P$.
\end{lemma}

\begin{proof}
We have to check~\eqref{Averaging} in both cases.

a) If $a, b\in M_0(X)$, then the following expressions are equal:
$$
T(a)T(b) = P(a)P(b), \quad 
T(T(a)b) = P(P(a)b), \quad 
T(aT(b)) = P(aP(b)),
$$
since $P$ is an averaging operator.
If $a, b \in F 1$, then 
$$
T(a)T(b) = \lambda^2, \quad 
T(T(a)b) = \lambda T(b) = \lambda^2, \quad 
T(aT(b)) = \lambda T(a) = \lambda^2.
$$
Considering $a \in M_0(X)$, $b \in F1$ in~\eqref{Averaging},
we obtain the required equalities:
\begin{gather*}
T(b)T(a) = \lambda b T(a), \quad
T(aT(b)) = T(\lambda b a) = \lambda b T(a), \\
T(T(a)b) = b T^2(a) = b P^2(a) = \lambda b P(a) = \lambda b T(a).
\end{gather*}

b) Since $F_0[X]$ is $T$-invariant,
it is easy to verify that the relations~\eqref{Averaging} hold for~$P$.
We also have by Lemma~\ref{Averaging_R(1)_in_F} that 
$T(1) = \lambda \in F$ 
and $P^2 = (T|_{F_0[X]})^2 = \lambda T|_{F_0[X]} = \lambda P$.
\end{proof}

In two lemmas below we find conditions
which allow us to define an RB-operator or an averaging 
one by a given linear monomial operator.

Let us introduce some notation.
Let $I$ be a well-ordered set of indexes and let $i_0$ be the smallest element from~$I$.
Under $\mathbb{N}^I$ we mean the set
$\{ (a_i)_{i\in I} \mid a_i \in \mathbb{N}\}$.
For any nonempty set $U$ and an arbitrary function $f \colon U \to \mathbb{N}^I$
satisfying $\sum_{i \in I} (f(u))_i < \infty$
we denote $f_{sum}(u) = \sum_{i \in I} (f(u))_i$, where $u \in U$.
We also use the following notation: $x^{f(u)} = \prod_{i\in I} x_i^{(f(u))_i}\in M(X)$.
Given a monomial $z\in M(X)$, we consider its multidegree in a standard way:
$\mathrm{mdeg}(z) = (\deg(x^i))_{i\in I}$.

Let $P$ be a linear monomial operator on $F_0[X]$ ($F[X]$).
Define by $P$ the function $\varphi \colon M_0(X) \to \mathbb{N}^I$ as follows:
\begin{equation} \label{RB_by_lin_op_phi}
\varphi(t) = 
\begin{cases}
\mathrm{mdeg}(P(t)), & t \not\in \ker P \cap M_0(X), \\
\mathrm{mdeg}(t), & t \in \ker P \cap M_0(X).
\end{cases}
\end{equation}
Working with $F[X]$, we additionally define $\varphi$ on~1:
\begin{equation} \label{RB_by_lin_op_phi(1)}
\varphi(1) = 
\begin{cases}
\mathrm{mdeg}(P(1)), & 1\not\in \ker P, \\
(r_i)_{i\in I}, & 1\in \ker P,
\end{cases}
\,\,
\mbox{ where }
r_{i} = \begin{cases}
1, & i = i_0, \\
0, & i\neq i_0.
\end{cases}
\end{equation}

\begin{lemma} \label{RB_by_lin_op_lemma}
Let $F$ be a field of characteristic zero,
$X = \{ x_i \mid i \in I\}$, $I \neq \emptyset$, 
and let $P$~be a linear monomial operator on $F[X]$
such that $\ker P\cdot\Imm P\subseteq \ker P$.
Additionally, we assume that $1 \not \in\Imm P$ 
and for any $u, v \in M(X)$, $P(v)\neq0$, if $uP(v)\in\ker P$, then $u\in \ker P$.
Suppose that the function $\varphi \colon M(X) \to \mathbb{N}^I$
defined by~\eqref{RB_by_lin_op_phi} and~\eqref{RB_by_lin_op_phi(1)}
satisfies the following relations:
\begin{equation}\label{EqForDeg}
\varphi(u) + \varphi(v) = \varphi(u x^{\varphi(v)}) = \varphi(x^{\varphi(u)}v),
\quad u, v \in M(X).
\end{equation}
Then the operator $R$ defined as follows:
\begin{equation}\label{RB_by_lin_op_lemma_formula}
R(t) = 
\begin{cases}
x^{\varphi(t)}/\varphi_{sum}(t), & t \not\in \ker P \cap M(X), \\
0, & t \in \ker P \cap M(X),
\end{cases}
\end{equation}
is an RB-operator of weight zero on $F[X]$.
In case of $F_0[X]$, it is enough to replace $M(X)$ by $M_0(X)$ above
and define $\varphi$ only by~\eqref{RB_by_lin_op_phi}.
\end{lemma}

\begin{proof}
At first, note that $\varphi_{sum}(t) \neq 0$ for any $t \in M(X)$.
Suppose that $\varphi_{sum}(t) = 0$, then $\varphi(t) = (0)_{i\in I}$, i.\,e. $P(t) \in F1$. 
By the conditions of Lemma, we have $1 \not\in \Imm P$ and so $\varphi(t) = \mathrm{mdeg}(t)$.
Hence $t = 1$ and $\varphi_{sum}(t) = 1$ by~\eqref{RB_by_lin_op_phi(1)}, a~contradiction.

Let us check that the relations~\eqref{RBO} hold for arbitrary $u, v \in M(X)$.
If $u, v~\not \in~\ker P$, then $x^{\varphi(u)}v, ux^{\varphi(v)} \not \in \ker P$; 
applying~\eqref{EqForDeg}, we deduce
\begin{multline*}
R(R(u)v + uR(v)) =
R\left(
\frac{x^{\varphi(u)}v}{\varphi_{sum}(u)} +
\frac{ux^{\varphi(v)}}{\varphi_{sum}(v)}
\right) \\ = 
\left(
\frac{1}{\varphi_{sum}(u)} +
\frac{1}{\varphi_{sum}(v)}
\right) 
\cdot
\frac{x^{\varphi(u) + \varphi(v)}}{\varphi_{sum}(u) + \varphi_{sum}(v)} = 
\frac{x^{\varphi(u)}}{\varphi_{sum}(u)} \cdot \frac{x^{\varphi(v)}}{\varphi_{sum}(v)} =
R(u)R(v).
\end{multline*}
If $u, v \in \ker P$, then the equalities~\eqref{RBO} are satisfied, since they have the form $0 = 0$.
Let us consider the remaining case, when $u \not\in \ker P$ and $v \in \ker P$.
In this case we have
$$
R(u)R(v) = 0, \quad
R\left(0\cdot v + \frac{ux^{\varphi(v)}}{\varphi_{sum}(v)}\right) =
\frac{R(ux^{\varphi(v)})}{\varphi_{sum}(v)} = 0,
$$
since $x^{\varphi(v)} \in \Imm P$ and $ux^{\varphi(v)} \in \ker P$ 
by the conditions of Lemma.
\end{proof}

\begin{lemma}\label{Averaging_by_lin_op_lemma}
Given $X \neq \emptyset$, let $P$ be a monomial linear operator on $F[X]$ ($F_0[X]$).
If for all $u, v \in M(X)$ ($M_0(X)$) 
there exists $w_{u, v} \in M(X)$ ($M_0(X)$) 
such that $P(u)P(v)$, $P(P(u)v), P(uP(v)) \in \Span(\{w_{u, v}\})$
and these three expressions are zero or nonzero simultaneously,
then the operator $T$ defined as
\begin{equation} \label{Averaging_by_lin_op_lemma_formula}
T(w) = 
\begin{cases}
P(w) / \alpha_w, & P(w) = \alpha_w t_w, \\
0, & P(w) = 0,
\end{cases}
\end{equation}
where $w \in M(X)$ ($M_0(X)$) and $t_w \in M(X)$ ($M_0(X)$),
is a monomial averaging operator on $F[X]$ ($F_0[X]$).
\end{lemma}

\begin{proof}
Let $u, v \in M(X)$ ($M_0(X)$) be monomials such that 
$P(u) = \alpha_u t_u$ and $P(v) = \alpha_v t_v$, $\alpha_u, \alpha_v \neq 0$.
Then $P(u)P(v) = \alpha_u \alpha_v t_u t_v$ and
$P(P(u)v)$, $P(uP(v)) \in \Span(\{t_u t_v\})$.
By the conditions of Lemma, we have $P(u)v, uP(v) \not\in\ker P$, hence
\begin{gather*}
T(u)T(v) = \dfrac{P(u)P(v)}{\alpha_u \alpha_v} = t_u t_v, \quad
T(T(u)v) = \dfrac{P(P(u)v)}{\alpha_u \alpha_{t_u v}} = t_u t_v, \\
T(uT(v)) = \dfrac{P(uP(v))}{\alpha_v \alpha_{u t_v}} = t_u t_v.
\end{gather*}
If one of the expressions $P(u)P(v)$, $P(P(u)v)$, or $P(uP(v))$ 
is zero, then all of them are zero.
Thus, $T(u)T(v) = T(T(u)v) = T(uT(v)) = 0$.
\end{proof}

\begin{remark}
In~\cite{Burde} (see the proof of Corollary 3.5)
one can find a similar idea of constructing RB-operators of weight zero
via invertible derivations.
\end{remark}

It is important to note that one may take~$P$ in Lemmas~\ref{RB_by_lin_op_lemma}
and~\ref{Averaging_by_lin_op_lemma} as an averaging operator or a Rota---Baxter one, respectively.
Thus we may switch between monomial averaging operators 
and Rota---Baxter ones.

\begin{remark} \label{averaging_operator_remark1}
Let $X$ be a nonempty set and let $F$ be a field of characteristic zero.
If $T$~is a monomial averaging operator on $F[X]$ or $F_0[X]$
such that $1\not \in \Imm T$, then the conditions of Lemma~\ref{RB_by_lin_op_lemma} for $T$ are satisfied 
by Lemma~\ref{closure_and_modules_lemma}.
\end{remark}

\begin{remark}\label{averaging_operator_remark2}
Given $X \neq \emptyset$. Let $R$ be a monomial RB-operator 
of weight zero on $F[X]$ ($F_0[X]$).
If the expressions $R(u)R(v), R(R(u)v), R(uR(v))$ 
are zero or nonzero simultaneously for all pairs $u, v \in M(X)$ ($M_0(X)$),
then $R$ satisfies the conditions of Lemma~\ref{Averaging_by_lin_op_lemma}.
\end{remark}

\begin{remark}\label{averaging_operator_remark3}
Given $X \neq \emptyset$. Let $P$ be a monomial linear operator on $F[X]$ ($F_0[X]$).
Suppose that $P(u)P(v), P(P(u)v), P(uP(v))$ are proportional
for all pairs $u, v \in M(X)$ ($u,v\in M_0(X)$).
On the one hand, this is enough to construct an averaging operator
by Lemma~\ref{Averaging_by_lin_op_lemma}.
On the other hand, by~\eqref{Averaging}, it is a necessary condition 
to obtain an averaging operator applying Lemma~\ref{Averaging_by_lin_op_lemma}.
\end{remark}

Let us define $\mathcal{R}(X)$ ($\mathcal{R}_0(X)$)
as the set of all monomial RB-operators of weight zero on $F[X]$ ($F_0[X]$).
Similarly, we define the set $\mathcal{T}(X)$~($\mathcal{T}_0(X)$)
of monomial averaging operators on $F[X]$ ($F_0[X]$).
Let us gather the results of Lemmas~\eqref{RB_F0[X]<->F[X]}--\eqref{Averaging_by_lin_op_lemma} 
in the following diagram:
\begin{center}
\begin{tikzcd}
\mathcal{T}(X)
\arrow[rr, "\eqref{RB_by_lin_op_lemma_formula}", shift left] 
\arrow[dd, "{P = T|_{F_0[X]}}"', shift right]
& & 
\mathcal{R}(X) 
\arrow[ll, "\eqref{Averaging_by_lin_op_lemma_formula}", shift left] 
\arrow[dd, "{P = R|_{F_0[X]}}", shift left] \\
& & \\
\mathcal{T}_0(X) 
\arrow[rr, "\eqref{RB_by_lin_op_lemma_formula}", shift left] 
\arrow[uu, "\eqref{Averaging_F0[X]<->F[X]_formula}"', shift right] 
& & 
\mathcal{R}_0(X) 
\arrow[ll, "\eqref{Averaging_by_lin_op_lemma_formula}", shift left] 
\arrow[uu, "\eqref{RB_F0[X]<->F[X]_formula}", shift left]                   
\end{tikzcd}
\end{center}
\noindent 
Note that there are some restrictions on operators
in all involved Lemmas~\eqref{RB_F0[X]<->F[X]}--\eqref{Averaging_by_lin_op_lemma}.

Now, we formulate the following two problems.

\begin{problem} \label{MainProblem}
Given an operator $T \in \mathcal{T}(X)$ ($\mathcal{T}_0(X)$),
describe all operators $R \in \mathcal{R}(X)$ ($\mathcal{R}_0(X)$)
having the form
\begin{equation} \label{MainProblemEquation}
R(z) = \alpha_z T(z), \,\, z \in M(X) \mbox{ ($M_0(X)$)},\,\, \alpha_z \in F.
\end{equation}
\end{problem}

The similar problem stated for monomial RB-operators of nonzero weight on $F[x,y]$ coming from homomorphic averaging operators was studied in~\cite{Khodzitskii}.

\begin{problem} \label{MainProblem2}
Given an operator $R \in \mathcal{R}(X)$ ($\mathcal{R}_0(X)$),
describe all operators $T \in \mathcal{T}(X)$ ($\mathcal{T}_0(X)$) having the form
\begin{equation}\label{MainProblemEquation2}
T(z) = \alpha_z R(z), \,\, z \in M(X) \mbox{ ($M_0(X)$)},\,\, \alpha_z \in F.
\end{equation}
\end{problem}

Define the set of all RB-operators of weight zero on $F[X]$~($F_0[X]$)
which have the form~\eqref{MainProblemEquation} for some monomial averaging operator~$T$ as $\mathcal{RT}(X)$ ($\mathcal{R T}_0(X)$).
Analogously, we define the set $\mathcal{T R}(X)$ ($\mathcal{T R}_0(X)$)
of averaging operators of the form~\eqref{MainProblemEquation2}.
Let us explain the notation: the first letter denotes the set of operators 
that we have at the output and the second letter denotes the set of operators 
that we have at the input.
For example, $\mathcal{RT}(X)$ consists of monomial operators from $\mathcal{R}(X)$ constructed by operators from $\mathcal{T}(X)$.

The main advantage of considering an RB-operator 
or an averaging one satisfying~\eqref{MainProblem} or~\eqref{MainProblem2} respectively 
is the following: the identities~\eqref{RBO} and~\eqref{Averaging} have to be verified only for scalars, i.\,e. we deal with recurrence relations on $\alpha_z \in F$.

In the light of Problems~\ref{MainProblem} and~\ref{MainProblem2}, 
we ask about the equality of corresponding sets of operators:

\begin{question} \label{AdditionalProblem}
Let $X$ be a nonempty set.

a) Is it true that $\mathcal{R}(X) = \mathcal{R T}(X)$
($\mathcal{R}_0(X) = \mathcal{R T}_0(X)$)?

b) Is it true that $\mathcal{T}(X) = \mathcal{T R}(X)$
($\mathcal{T}_0(X) = \mathcal{T R}_0(X)$)?
\end{question}

In the current work we focus on Question~\ref{AdditionalProblem}.
This allows us to clarify the difference between
monomial RB-operators and monomial averaging ones defined on the polynomial algebra.

\begin{remark}
If $T \in \mathcal{T}(X)$ ($\mathcal{T}_0(X)$)
and $1 \in \Imm T$, then we have $\ker T \subsetneq \ker R$
for any RB-operator $R$ defined by~\eqref{MainProblemEquation}.
Indeed, since $1 \in \Imm T$, there exists $z \in M(X)$ ($M_0(X)$)
such that $T(z) \in F^*$. 
Applying~\eqref{RBO} for $R(z)R(z)$, 
we obtain $\alpha_z^2(T(z))^2 = 2\alpha_z^2(T(z))^2$, hence $\alpha_z = 0$.
\end{remark}

\begin{example} \label{RB_bot_nor_averaging_example}
The following operator is an RB-operator of weight zero on $F_0[x, y]$
for any $n > 1$, $\beta_1,\ldots,\beta_{n-1}\in F$,
and $\beta_n \in F^*$: 
$$
R(x^{sn + k} y^{m}) = 
\begin{cases}
k\beta_k x^{sn} y^{m + k} / m + k , & n > k > 0, \\
\beta_n x^{sn} / s, & k = m = 0, \\
0, & k = 0,\,\, m > 0.
\end{cases}
$$
\end{example}

According to Remark~\ref{averaging_operator_remark3},
we can not construct an averaging operator~$T$,
using Lemma~\ref{Averaging_by_lin_op_lemma} for the RB-operator $R$ 
from Example~\ref{RB_bot_nor_averaging_example}.
It is so, since we have 
\begin{gather*}
R(x^{an + b})R(x^{cn}) = 
R(x^{an + b}R(x^{cn})) =
\alpha_n \alpha_b x^{(a + c)n} y^b, \quad
R(R(x^{an + b}) x^{cn}) = 0
\end{gather*}
for $0 < a, c$ and $0 < b < n$.
There exists a similar example on $F_0[x, y]$.

\vfill \eject

However, we can construct an averaging operator defined by~\eqref{MainProblemEquation2} 
for this~$R$ assuming $\alpha_{sn, 0} = 0$ for all $s > 0$.
So, $\ker R \subsetneq \ker T$ and there is no way to construct 
an operator $T$ without extension the kernel of $R$.

\begin{remark}
If $R \in \mathcal{R}(X)$ ($\mathcal{R}_0(X)$)
and $R(z)w \in \ker R$ for some $z, w \in M(X)$ ($M_0(X)$) such that $z, w \not \in \ker R$,
then we have $\ker R \subsetneq \ker T$ for any averaging operator $T$ defined by~\eqref{MainProblemEquation2}.
Indeed, applying the equality $T(z)T(w) = T(T(z)w)$
and the conditions on the choice of $z,w$, we obtain $R(z)R(w) \neq 0$, $R(R(z)w) = 0$,
hence $\alpha_z \alpha_w = 0$, i.\,e. $z \in \ker T$ or $w \in \ker T$.
\end{remark}

Throughout the work we will use the following notation
to write the corresponding set of operators:
$\mathcal{L}(x) := \mathcal{L}(\{x\})$ and
$\mathcal{L}(x, y) := \mathcal{L}(\{x, y\})$.

\section{Monomial averaging operators on $F[x]$ and $F_0[x]$}

Let us show that monomial RB-operators of weight zero 
and monomial averaging operators defined on a polynomial algebra are very close.
Firstly, we recall the known description of monomial RB-operators of weight zero on $F_0[x]$ and $F[x]$.

\begin{theorem}[\!\!\cite{Monom,MonomNonunital}]\label{RB_on_F[x]_and_F0[x]}
Let $F$ be a field of characteristic zero and
let $R$ be a monomial RB-operator of weight zero on $F_0[x]$ ($F[x]$). 
Then there exist $0 < m$, $p_i \in \mathbb{N}$, and $q_i \in F$, 
$i = 1,\dots, m$ ($i = 0,\dots, m - 1$), such that
$p_i = 0$ if and only if $q_i = 0$,
and $R$ has the form
$$
R(x^{m a+b})= q_b \frac{x^{m(a+p_b)}}{m(a+p_b)},
$$
where $a \in \mathbb{N}$ and $0<b \leqslant m$ ($0 \leqslant b<m$).
\end{theorem}

Now, we describe monomial averaging operators on $F_0[x]$ ($F[x]$), 
following the proof of Theorem~\ref{RB_on_F[x]_and_F0[x]} from~\cite{MonomNonunital}.
In the case of averaging operators, unlike of Rota---Baxter ones, 
we may deal with a field of any characteristic.

\begin{theorem}\label{Averaging_on_F[x]_and_F0[x]}
a) Let $T$ be a monomial averaging operator on $F_0[x]$.
Then there exist $0 < m$, $p_i \in \mathbb{N}$, and $q_i \in F$, 
$i = 1,\dots, m$, such that if $p_i = 0$, then $q_i = 0$,
and~$T$ has the form
$$
T(x^{m a + b})= q_b x^{m(a + p_b)},
$$
where $a \in \mathbb{N}$ and $0 < b \leqslant m$.

b) Let $T$  be a monomial averaging operator on $F[x]$.
Then either there exist $0 < m$, $p_i \in \mathbb{N}$ and $q_i \in F$, 
$i = 0,\dots, m - 1$, such that $T$ has the form
$$
T_1(x^{m a + b})= q_b x^{m(a + p_b)},
$$
where $a \in \mathbb{N}$ and $0 \leqslant b < m$,
or there exist $q_n \in F$, $0\leqslant n$, such that $T$ has the form
$$
T_2(x^n) = q_n.
$$
\end{theorem}

\begin{proof}
It is easy to see that the operators presented in the formulation of Theorem are indeed averaging ones.

For the proof of a), one can repeat the proof of Theorem~\ref{RB_on_F[x]_and_F0[x]}
in the case of a field~$F$ of an arbitrary characteristic.
For completeness we provide another proof in the case of a~field of characteristic zero, 
which, however, involves the result of Theorem~\ref{RB_on_F[x]_and_F0[x]}.

a) By Remark~\ref{averaging_operator_remark1}, 
we may apply Lemma~\ref{RB_by_lin_op_lemma} for~$T$.
So, we obtain an RB-operator~$R$ on $F_0[x]$ constructed by $T$.
Note that $x^n \in \ker T$ if and only if $x^n \in \ker R$.
Applying Theorem~\ref{RB_on_F[x]_and_F0[x]}, 
the operator $T$ necessarily takes the following form: 
$$ 
T(x^{ma + b}) 
= \begin{cases}
q_{a, b}x^{m(a + p_b)}, & x^{ma + b} \not\in \ker R, \\
0, & x^{ma + b} \in \ker R,
\end{cases}
$$
where $0<m$ is fixed, $q_{a,b} \neq 0$, $0 \leqslant a$, $ 0 < b \leqslant m$, and $0 < p_b$.
By~\eqref{Averaging}, we have the following restrictions on coefficients:
\begin{gather}\label{r_c,d=r_a+c+pb+d}
q_{c, d} = q_{a + c + p_b, d}, 
\,\, 0 \leqslant a, c, \,\, 0 < b, d \leqslant m.
\end{gather}
Introduce the following notation:
$$
P = \{ b \mid p_b > 0\}, 
\quad p_{\min} = \min_{i \in P} p_i.
$$
If $P = \emptyset$, then $R = 0$ and so $T = 0$.
Let $P \neq \emptyset$.
We consider~\eqref{r_c,d=r_a+c+pb+d} 
with $a = 0$, $d\in P$, $b$ such that $p_b = p_{\min}$
and obtain $q_{c, d} = q_{c + p_{\min}, d}$ for all $c\geqslant 0$. 
Let us take~$l$ from the interval $1 \leqslant l \leqslant p_{\min}$. 
Applying the last formula, we derive 
$q_{l, d} = q_{l + kp_{\min}, d}$ for all $d \in P$ and $k \geqslant 0$.
Substituting $p_b = p_{\min}$ in~\eqref{r_c,d=r_a+c+pb+d}, we obtain
$q_{c, d} = q_{a + c + p_{\min}, d}$, hence for $c = 1$ and any $1 \leqslant a < p_{\min} - 1$
we obtain $q_{1, d} = q_{a + 1 + p_{\min}, d} = q_{a + 1, d}$.
Therefore, $q_{i, d} = q_{j, d}$ for all $i, j \geqslant 0$.
We have received the required operator.

b) Denote $T(x^n) = q_n x^{t_n}$, where $q_n \in F$ and $t_n \geqslant 0$
for $n \geqslant 0$.
Let us exclude the case when $\Imm T \subseteq F$,
such an operator coincides with $T_2$.

Let us prove the following formula for $k \geqslant 0$ by induction on~$k$:
\begin{equation}\label{xn^k+1=xn^k_and_xn}
    (T(x^n))^{k + 1} = T((T(x^n))^k x^n).
\end{equation}
The base case $k = 0$ is trivial.
The inductive step follows from the equalities:
$$
(T(x^n))^kT(x^n) \stackrel{\eqref{Averaging}}{=}
T((T(x^n))^{k - 1} x^n)T(x^n) \stackrel{\eqref{xn^k+1=xn^k_and_xn}}{=}
T((T(x^n))^k x^n).
$$

Let $s$ be such that $T(x^s) = q_s x^{t_s}$, $q_s \neq 0$ and $t_s > 0$.
Then $m = \gcd (\{t \mid T(x^t) \not \in F\})$ is correctly defined,
i.\,e. for any $x^r \in \Imm T$ we have $r = am$ for some $a \geqslant 0$. 
On the other hand, there exists $N \geqslant 0$ such that
$\Imm T \supset \Span (\{x^{m k} \mid k \geqslant N \})$.
We will use the following notation throughout the proof:
$q_{a, b} = q_{ma + b}$, $t_{a, b} = t_{ma + b}$,
where $T(x^{ma + b}) = q_{ma + b} x^{mt_{ma + b}}$.

Consider $0 \leqslant b < m$ such that $T(x^{ma + b}) = 0$ for some $a \geqslant 0$.
Now we prove that $T(x^{mc + b}) = 0$ for all $c \geqslant 0$.
By Lemma~\ref{closure_and_modules_lemma}, $\ker T$ is an $\Imm T$-bimodule,
it implies that $T(x^{mc + b}) = 0$ for all $c \geqslant a + N$.
Suppose that $T(x^{mc + b}) = q_{c, b} x^{md}$, $q_{c, b} \neq 0$, $d > 0$,
for some $c < a + N$.
By~\eqref{xn^k+1=xn^k_and_xn}, we have
$$
q_{c, b}^{k + 1} x^{md(k + 1)}  = (T(x^{mc + b}))^{k + 1} = 
T((T(x^{mc + b}))^k x^{mc + b}) = T(q_{c, b}^k x^{m(dk + c) + b}).
$$
Taking $k \geqslant 0$ such that $dk + c \geqslant N + a$, we obtain a contradiction.
We also need to consider the case when $T(x^{mc + b}) = q_{c, b} \in F\setminus \{0\}$.
Let $k \geqslant 0$ be such that $T(x^k) = \gamma x^{m(a + N)}$, where $\gamma \neq 0$.
Then by~\eqref{Averaging} we again obtain a contradiction:
$$
0 \neq 
q_{c, b}\gamma x^{m(a + N)} = 
T(x^{mc + b})T(x^k) = 
T(x^{mc + b}T(x^k)) =
T(x^{m(c + a + N) + b}) = 0.
$$
Thus $T(x^{ma + b}) = 0$ for some $a \geqslant 0$
if and only if $T(x^{mc + b}) = 0$ for every $c \geqslant 0$.

Let $0 \leqslant b < m$ be such that $T(x^b) = q_{0, b} x^{mt_{0, b}}$, 
$q_{0, b} \neq 0$, $t_{0, b} \geqslant 0$.
We want to prove that $R(x^{ma + b}) = q_{0, b} x^{m(a + t_{0, b})}$ for all $a \geqslant 0$.
Consider the case $a \geqslant N$ at first. 
Since $x^{ma} \in \Imm T$, there exists $k$ such that $T(x^k) = \gamma x^{ma}$, $\gamma \neq 0$.
Hence
$$
q_{0, b}\gamma x^{m(a + t_{0, b})} = T(x^b)T(x^k) = T(x^b T(x^k)) = \gamma T(x^{ma + b}).
$$
Now we consider the case $0 < a < N$. 
We have $T(x^{ma + b}) = q_{a, b} x^{m r}$ 
for some $r \geqslant 0$ depending on $a$.
From the previous part of the proof we know that $q_{a,b} \neq 0$.
The following expressions have to be equal to each other,
since $T$ is an averaging operator:
\begin{gather} \label{Averaging_on_F[x]_and_F0[x]_last_eq}
\begin{gathered}
T(x^b)T(x^{ma + b}) = q_{0, b}q_{a, b} x^{m(t_{0, b} + r)}, \quad 
T(T(x^b)x^{ma + b}) = q_{0, b} T(x^{m(a + t_{0, b}) + b}), \\
T(x^b T(x^{ma + b})) = q_{a, b} T(x^{mr + b}). 
\end{gathered}
\end{gather}
So, we obtain $T(x^{m(a + t_{0, b}) + b}) = q_{a, b} x^{m(t_{0, b} + r)}$ and
$T(x^{mr + b}) = q_{0, b} x^{m(t_{0, b} + r)}$.
We have already considered the case $r = a + t_{0, b}$.
Suppose that $r \neq a + t_{0, b}$.
By~\eqref{Averaging_on_F[x]_and_F0[x]_last_eq},
$q_{0, b}x^{m(a + t_{0, b}) + b} - q_{a, b}x^{mr + b} \in \ker T$.
Multiplying this expression by $x^{mN}\in \Imm T$ and then acting by~$T$, we get
$$
0 = T(q_{0, b}x^{m(a + t_{0, b} + N) + b} - q_{a, b}x^{m(r + N) + b}) 
 = q_{0, b}^2 x^{m(a + 2t_{0, b} + N)} - q_{a, b}q_{0, b} x^{m(r + t_{0, b} + N)}.
$$
It is a contradiction, since $q_{0, b}, q_{a, b} \neq 0$
and the degrees of the monomials are not equal.

The operator $T$ is completely defined by its action on monomials $x^a$,
$0 \leqslant a < m$, and $T$ coincides with $T_1$.
\end{proof}

In the light of Lemmas~\ref{RB_by_lin_op_lemma}
and~\ref{Averaging_by_lin_op_lemma}, the description
of monomial averaging operators and RB-operators on $F_0[x]$ 
coincide up to division by the denominator.
On $F[x]$ we have a difference between monomial RB-operators and monomial averaging ones, 
since the unit may lie in the image of an averaging operator,
but not in the image of an RB-operator.

In particular, over a field of characteristic zero
Question~\ref{AdditionalProblem} has a positive answer in a) and b) in the case of $F_0[x]$.
The answer to a) is positive and to b) is negative in the case of $F[x]$.
At the same time, if we exclude from the set~$\mathcal{T}(x)$ all monomial averaging operators
having the unit in its image, then for this subset the question analogous to~\ref{AdditionalProblem}
has a positive answer in a) and b).

\section{Non-increasing operators on $F_0[x, y]$}

In the current section we study Problem~\ref{AdditionalProblem}
on the polynomial algebra in two variables with some restrictions on operators.
 
Below, we will use the notation $k \equiv_n l$
to denote the number $k$ is congruent to $l$ by modulo $n$.

\begin{definition}
Let $A = \bigoplus\limits_{n \in \mathbb{Z}} A_n$ 
be a graded algebra and let $L$ be a linear operator on $A$.
We say that $L$ is non-increasing (in degree) if
$L(A_i) \subseteq \bigoplus\limits_{j \leqslant i} A_j$ for all $i \in \mathbb{Z}$.
\end{definition}

A linear monomial operator $L$ on $F[x, y]$ ($F_0[x, y]$)
is non-increasing (in degree) if for the action 
$L(x^n y^m) = \alpha_{n, m} x^s y^t$,
where $\alpha_{n, m} \in F$, $n, m \geqslant 0$ ($n + m > 0$),
we have $s + t \leqslant n + m$.

Let us set the general property that a monomial RB-operator of weight zero on
the polynomial algebra can not map monomials to monomials of lower degree arbitrarily.

\begin{lemma} \label{lem:ZeroDecreaseDegree}
Let $X = \{ x_i \mid i \in I\}$, $I \neq \emptyset$,
and let $R$ be a monomial RB-operator of weight zero on $F_0[X]$ ($F[X]$).
Then for any nonempty subset $J \subseteq I$ the following can not happen on $F[X]$:
$R\big(\prod_{i \in J} x^{a_i}_i\big) = \alpha \prod_{i \in J} x^{b_i}_i$, where
$\prod_{i \in J} x^{a_i}_i \in M_0(X)$ ($M(X)$), 
$\alpha \neq 0$, $a_i \geqslant b_i \geqslant 0$ for all $i \in J$, and there exists 
$k \in J$ such that $a_k > b_k$.
\end{lemma}

\begin{proof}
Assume the converse, let
$R\big(\prod_{i \in J} x^{a_i - b_i}_i\big) 
= \beta \prod_{i \in J} x^{c_i}_i$, $\beta \in F$. 
By~\eqref{RBO} we have
\begin{multline*}
R\Bigl(\prod_{i \in J} x^{a_i}_i\Bigl) 
R\Bigl(\prod_{i \in J} x^{a_i - b_i}_i\Bigl) = 
\alpha \beta \prod_{i \in J} x^{b_i}_i \cdot 
\prod_{i \in J} x^{c_i}_i \\ =
R\Bigl(\alpha \prod_{i \in J} x^{b_i}_i \cdot 
\prod_{i \in J} x^{a_i - b_i}_i + 
\beta \prod_{i \in J} x^{a_i}_i 
\prod_{i \in J} x^{c_i}_i\Bigl) =
\alpha^2 \prod_{i \in J} x^{b_i}_i + 
\beta R\Bigl(\prod_{i \in J} x^{a_i + c_i}_i\Bigl).
\end{multline*}
If $\beta = 0$, then we arrive at a~contradiction, since $\alpha \neq 0$.
If $\beta \neq 0$, then 
$
\prod_{i \in J} x^{b_i + c_i}_i =
\prod_{i \in J} x^{b_i}_i
$
and~$c_i = 0$ for all $i \in J$, i.\,e. $1\in \Imm R$.
In the case of $F_0[X]$, we have no unit, a~contradiction.
In the case of $F[X]$, this contradicts Lemma~\ref{RB_R(1)_in_F}.
\end{proof}

\begin{remark}  \label{remark:ZeroDecreaseDegree}
Let $X$ be a nonempty set.
A similar to Lemma~\ref{lem:ZeroDecreaseDegree} result for 
monomial averaging operators holds only on $F_0[X]$.
In the case of $F[X]$, the analogue of Lemma~\ref{lem:ZeroDecreaseDegree} 
does not hold, since for polynomials even in one variable
there exists an averaging operator~$T$ such that 
$T(x^n) = \alpha x^{n - c}$ for some $0 < c \leqslant n$ and $\alpha \in F\setminus \{ 0 \}$
(see Theorem~\ref{Averaging_on_F[x]_and_F0[x]}).
\end{remark}

We can also provide an averaging operator on the algebra $F[x, y]$
for which the analogue of Lemma~\ref{lem:ZeroDecreaseDegree} does not hold.

\begin{example}
The following operator is an averaging one on $F[x, y]$:
$$
T(x^n y^m) = 
\begin{cases}
x^{2s} y^m, & n = 2s + 1, \, m > 0, \\
x^{2s} y^m, & n = 2s, \, m \geqslant 0.
\end{cases}
$$
\end{example}

Let us describe monomial RB-operators of weight zero and averaging operators
that are non-increasing and do not have monomials in the kernel.

Note that the description of monomial non-increasing RB-operators of weight zero on $F[x,y]$
follows from the solution of the same problem formulated on $F_0[x,y]$.
To show this, we apply the correspondence from Lemma~\ref{RB_F0[X]<->F[X]}, which is one-to-one.
The latter is true, since for any monomial non-increasing RB-operator~$R$ of weight zero on $F[x, y]$
we have $R(F) \subseteq F$ and $R(1) = 0$ by Lemma~\ref{RB_R(1)_in_F}.
However, there exist non-increasing RB-operators~$R$ on $F_0[x, y]$ such that $R^2 \neq 0$.

Given a monomial non-increasing averaging operator~$T$ on $F[x,y]$,
it may happen that $F_0[x,y]$ is not $T$-invariant, e.\,g. when $1 \in \Imm T$.
Thus, the correspondence from Lemma~\ref{Averaging_F0[X]<->F[X]}
is not one-to-one for monomial non-increasing averaging operators.

\begin{example} \label{not_increase_degree_operator_example}
An operator $T$ on $F[x,y]$ defined by the formula
$T(x^n y^m) = \alpha_n y^m$, where $\alpha_n \in F$, $n, m \geqslant 0$,
is a~monomial non-increasing averaging operator.
If $\alpha_i \neq 0$ for some $i > 0$, 
then there exists $z \in M(\{x,y\})$, $\deg z > 0$, such that $T(z) \in F^*$.
\end{example}

We can not obtain an averaging operator on $F_0[x, y]$ 
applying Lemma~\ref{Averaging_F0[X]<->F[X]} to the operator $T$ from Example~\ref{not_increase_degree_operator_example},
since $F_0[x, y]$ is not $T$-invariant.
We also can not use Lemma~\ref{RB_by_lin_op_lemma} if $1 \in \Imm T$.
To apply both mentioned above statements, 
we will describe averaging operators~$T$ on $F[x,y]$ such that $F_0[x,y]$ is $T$-invariant.

\begin{lemma} \label{RB_u->u,v->v}
Let $A$ be an algebra over a field~$F$ of characteristic zero and let $R$ be an RB-operator of weight zero on $A$.
For any $u,v\in A$ such that $R(u) = \alpha u$, $R(v) = \beta v$, $\alpha, \beta\in F^*$,
the following formula holds:
\begin{equation*}
R(u^n v^m) = \dfrac{\alpha\beta}{m\alpha + n\beta}u^n v^m.
\end{equation*}
\end{lemma}

\begin{proof}
The equalities $R(u^k) = \alpha u^k / k$ and $R(v^k) = \beta v^k / k$
for $k>0$ follow from~\eqref{RBO} by induction.
Considering~\eqref{RBO} with $R(u^n)R(v^m)$, we obtain
$$
\left(\frac{\alpha}{n} + \frac{\beta}{m}\right)R(u^n v^m) 
 = \frac{\alpha\beta}{nm}u^n v^m, \,\, m + n > 0,\,\, m, n \geqslant 0.
$$
If $m\alpha + n\beta = 0$, then $\alpha\beta = 0$, a contradiction with $\alpha, \beta \neq 0$.
Hence $m\alpha + n\beta \neq 0$. Dividing on this expression, we prove the required formula.
\end{proof}

Let us define an automorphism $\psi_{x,y}$ of $F[x,y]$ ($F_0[x, y]$) as
$\psi_{x,y}(x) = y$, $\psi_{x,y}(y) = x$.

\begin{theorem} \label{RB_not_incr_and_ker=0}
Let $F$ be a field of characteristic zero
and let $R$ be a monomial non-increasing RB-operator of weight zero on $F_0[x, y]$ 
that does not have monomials in the kernel.
Then $R$ up to conjugation with $\psi_{x, y}$
coincides with one of the operators $R_1,R_2$, where
\begin{equation} \label{RB_not_incr_and_ker=0_solution1}
R_1(x^{sn + k} y^m) = 
\begin{cases}
\frac{(k - a_k)\beta_k\beta_n}{s\beta_0 + (m + k - a_k)\beta_n}x^{sn} y^{m + k - a_k}, 
& 0 < k < n, \, 0 \leqslant s, m,\\
\frac{\beta_0\beta_n}{s\beta_0 + m\beta_n}x^{sn} y^m, 
& k = 0, \, 0 \leqslant s, m,\, 0 < s + m,
\end{cases}
\end{equation}
$0 < n$,
$0 \leqslant a_k < k$ for $0 < k < n$
and $\beta_0, \ldots, \beta_{n - 1} \in F^*$;
\begin{equation} \label{RB_not_incr_and_ker=0_solution2}
R_2(x^k y^m) = 
\begin{cases}
\frac{(k - a_k)\beta_k}{m + k - a_k}y^{m + k - a_k}, & k > 0, \, m \geqslant 0,\\
\frac{\beta_0}{m}y^m, & k = 0, \, m > 0,
\end{cases}
\end{equation}
$0 \leqslant a_k < k$ for $0 < k$ and $\beta_i \in F^*$ for $i \geqslant 0$.
\end{theorem}

\begin{proof}
One may check that the operators $R_1, R_2$ are indeed RB-operators of weight zero.

By the conditions, $R(x^s y^t) = \beta_{s, t} x^{a_{s,t}} y^{b_{s,t}}$, $\beta_{s, t} \neq 0$,
for any $x^s y^t \in M_0(\{ x, y\})$.

Consider the action of $R$ on monomials of the degree one.
Since $R$ is non-increasing, up to conjugation with $\psi_{x,y}$ we have only three cases.

{\sc Case 1}: $R(x) = \alpha y$ and $R(y) = \beta x$.
Consider the following expressions:
$$
R(x)R(x) = \alpha^2 y^2 = 2\alpha R(xy), \quad
R(y)R(y) = \beta^2 x^2 = 2\beta R(xy).
$$
So, $\alpha y^2 = \beta x^2$, a contradiction.

{\sc Case 2}: $R(x) = \alpha x$ and $R(y) = \beta y$.
By Lemma~\ref{RB_u->u,v->v}, we obtain the operator $R_1$ 
defined by~\eqref{RB_not_incr_and_ker=0_solution1} with $n = 1$.

{\sc Case 3}: $R(x) = \beta_1 y$, $R(y) = \beta_0 y$.
Due to Lemma~\ref{lem:ZeroDecreaseDegree}, 
$R$ can not map $x^n$ or $y^n$ to $x^{n - p}$ or $y^{n - p}$ 
for some $0 < p < n$, respectively.
We will apply this further.

{\sc Step 1}. 
Define $n = \min\{s\mid R(x^s) = \beta_s x^s, \,\, \beta_s \neq 0\}$.
If there are no $s$ such that $R(x^s) = \beta_s x^s$ and $\beta_s \neq 0$, then we put $n = \infty$.
Substituting $u = v = y$ in Lemma~\ref{RB_u->u,v->v}, 
we find that $R(y^m) = \beta_0 y^m / m$, $m > 0$.

{\sc Step 2}. 
Consider arbitrary $k < n$.
Then $R(x^k) = \beta_k x^{k - t} y^l$ for some $0 < t \leqslant k$ and $0 \leqslant l \leqslant t$.
We want to show that $t = k$ and hence $l > 0$.
Suppose that $k$ is a~minimal number such that 
$R(x^k) = \beta_k x^{k - t} y^l$ and $0 < t < k$.
By the choice of~$k$, we have $R(x^s) = \beta_s y^{s - a_s}$ for all $0 < s < k$
and some $0 \leqslant a_s < s$.
Applying~\eqref{RBO} for $R(x^s)$ and $R(y^m)$, where $0 < s < k$, $0 < m$,
we obtain the following formula:
\begin{gather} \label{all_in_y_decrease_degreeformula}
R(x^s y^m) = \frac{(s - a_s)\beta_s}{m + s - a_s}y^{m + s - a_s}, \quad m \geqslant 0.
\end{gather}
Now we rewrite~\eqref{RBO} with the help of~\eqref{all_in_y_decrease_degreeformula}:
\begin{multline*}
R(x^k)R(y^m) = 
(\beta_k \beta_0 / m) x^{k - t} y^{m + l} = 
R(\beta_k x^{k - t} y^{m + l} + (\beta_0 / m )x^k y^m) \\ = 
\frac{(k - t - a_{k - t}) \beta_k \beta_{k - t}}{m + l + k - t - a_{k - t}} 
y^{m + l + k - t - a_{k - t}} + (\beta_0 / m )R(x^k y^m).
\end{multline*}
We see that $R(x^k y^m)$ does not equal a monomial with some coefficient, it is a contradiction. 

{\sc Step 3}. 
We have proved that $R(x^s) = \beta_s y^{s - a_s}$, $0 \leqslant a_s < s$, for all $0 < s < n$.
If $n = \infty$, then the formula~\eqref{all_in_y_decrease_degreeformula} holds for all $s > 0$.
This equality taken jointly with the one $R(y^m) = \beta_0 y^m / m$, $m > 0$,
implies that $R$ coincides with $R_2$.

If $n < \infty$, then $R(x^n) = \beta_n x^n$
and $R(x^{sn}) = \beta_n x^{sn} / s$ by Lemma~\ref{RB_u->u,v->v} for all $s > 0$.
Again by Lemma~\ref{RB_u->u,v->v} we obtain the formula
$R(x^{sn} y^m) 
 = \frac{\beta_0\beta_n}{s\beta_0 + m\beta_n}x^{sn} y^m$.
Applying~\eqref{RBO} for $R(x^{an} y^m)$ and $R(x^k)$, where $0 < k < n$, $0 \leqslant a, m$,
we find that $R$ coincides with $R_1$.
\end{proof}

\begin{remark} \label{theorem_decrease_degree_fractions_remark_about}
Let $L$ be a linear operator defined by~\eqref{RB_not_incr_and_ker=0_solution1},
then $L$ is an RB-operator provided that $\beta_n / \beta_0 \not \in \mathbb{Q}_{<0}$.
\end{remark}

\begin{remark} \label{theorem_decrease_degree_generalization_remark}
We may take $a_k \in \mathbb{Z}$, $a_k < k$, and $\beta_k = 0$ in the formulas~\eqref{RB_not_incr_and_ker=0_solution1}
and~\eqref{RB_not_incr_and_ker=0_solution2}, allowing $k$ be equal to~0.
If we keep conditions from Remark~\ref{theorem_decrease_degree_fractions_remark_about},
then we also get RB-operators of weight zero on $F_0[x,y]$.
Under the conditions, if we have $\beta_0 = 0$, then $R^2 = 0$ and one may apply  Lemma~\ref{RB_F0[X]<->F[X]}.
\end{remark}

Let us consider some examples and generalizations of the RB-operators $R_1,R_2$
obtained in Theorem~\ref{RB_not_incr_and_ker=0}.

\begin{example}
We write down the operator~$R_1$ with $n = 10$, $a_k = k - 1$, $0 < k < 10$,
$\beta_{2k} \neq 0$, $0 \leqslant k \leqslant 5$, and $\beta_{2k + 1} = 0$, $0 \leqslant k \leqslant 4$:
$$
R_1(x^{10s + k} y^m) = 
\begin{cases}
\frac{\beta_k\beta_{10}}{s\beta_0 + (m + 1)\beta_{10}}x^{10s} y^{m + 1}, & 0 < k < 10, \, k \equiv_2 0 , \\
\frac{\beta_0\beta_{10}}{s\beta_0 + m\beta_{10}}x^{10s} y^m, & k = 0, \\
0 & \text{otherwise}.
\end{cases}
$$
\end{example}

\begin{example}
The generalization of the operator $R_2$ (see Remark~\ref{theorem_decrease_degree_generalization_remark}) 
with parameters $\beta_k = 1$ and $a_k = 0$ for $k > 0$ and $\beta_0 = 1$ has the following form:
$$
R_2(x^n y^m) = \begin{cases}
\frac{n}{n+m}y^{n+m}, & (n,m)\neq (0,0), \\
0, & n = m = 0.
\end{cases}
$$
This operator appeared in~\cite{Viellard-Baron}.
\end{example}

\begin{example}
Taking $a_k = -1$ in~\eqref{RB_not_incr_and_ker=0_solution1}
and~\eqref{RB_not_incr_and_ker=0_solution2},
we obtain the following RB-operators:
$$
R(x^{sn + k} y^m) = 
\begin{cases}
\frac{(k + 1)\beta_k\beta_n}{s\beta_0 + (m + k + 1)\beta_n}x^{sn} y^{m + k + 1}, & 0 < k < n, \\
\frac{\beta_0\beta_n}{s\beta_0 + m\beta_n}x^{sn} y^m, & k = 0,
\end{cases}
$$
$$
R(x^k y^m) = 
\begin{cases}
\frac{(k + 1)\beta_k}{m + k + 1}y^{m + k + 1}, & 0 < k, \\
\frac{\beta_0}{m}y^m, & k = 0.
\end{cases}
$$
\end{example}

Now we move on to the description of the similar class of averaging operators~on~$F_0[x, y]$.

\begin{theorem} \label{Averaging_not_incr_and_ker=0}
Let $T$ be a monomial non-increasing averaging operator on $F_0[x, y]$
that does not have monomials in the kernel.
Then $T$ up to conjugation with $\psi_{x, y}$
coincides with one of the operators $T_1, T_2$, where
$$
T_1(x^{sn + k} y^m) = \beta_k x^{sn} y^{m + k - a_k}, 
\,\, 0 \leqslant n, m, \,\, 0 \leqslant k < n, \,\, 0 < sn + k + m,
$$
$a_0 = 0$, $0 \leqslant a_k < k$ for $0 < k < n$
and $\beta_0, \ldots, \beta_{n - 1} \in F^*$;
$$
T_2(x^k y^m) = \beta_k y^{m + k - a_k}, 
\,\, 0 \leqslant k, m, \,\, 0 < k + m,
$$ 
$a_0 = 0$, $0 \leqslant a_k < k$ for $0 < k$ and $\beta_i \in F^*$ for $i \geqslant 0$.
\end{theorem}

\begin{proof}
The proof of the statement is analogous to the one of Theorem~\ref{RB_not_incr_and_ker=0}.
We will carry out the proof step by step,
excepting technical details, and pay attention to the main points,
where the specifics of averaging operators is important.

{\sc Case 1}: $T(x) = \alpha y$ and $T(y) = \beta x$.
We arrive at a contradiction with that $T$ is monomial:
$T(xy) = T(x)T(x) / \alpha = \alpha y^2$,
$T(xy) = T(y)T(y) / \beta = \beta x^2$.

{\sc Case 2}: $T(x) = \alpha x$ and $T(y) = \beta y$.
Then $T(x^n) = \alpha x^n$ and $T(y^m) = \beta y^m$
for all $n, m > 0$.
Due to~\eqref{Averaging}, we find that $\alpha = \beta$ 
and $T$ coincides with $T_1$ for $n = 1$.

{\sc Case 3}: $T(x) = \beta_1 y$, $T(y) = \beta_0 y$.
Hence $T(y^m) = \beta_0 y^m$ for any $m > 0$.

{\sc Step 1}. 
We define 
$n = \min\{s\mid R(x^s) = \beta_n x^s,\,\, \beta_s \neq 0\}$
and put $n = \infty$ if the mentioned set is empty.

{\sc Step 2}.
We want to prove that if $k < n$ and $T(x^k) = \beta_k x^{k - t} y^l$
for some $0 < t \leqslant k$ and $0 \leqslant l \leqslant t$, then $t = k$.
To the contrary, let $k$ be a minimal number such that
$T(x^k) = \beta_k x^{k - t} y^l$, $0 < t < k$.
In this case we derive the formula similar to~\eqref{all_in_y_decrease_degreeformula}:
$T(x^s y^m) = \beta_s y^{m + s - a_s}$, where $0 < s < k$ and $0\leqslant m$.
Due to~\eqref{Averaging}, the expressions 
\begin{gather*}
T(x^k)T(y^m) = \beta_k \beta_0 x^{k - t} y^{m + l}, \quad
T(x^kT(y^m)) = \beta_0 T(x^k y^m), \\
T(T(x^k)y^m) = \beta_k T( x^{k - t} y^{m + l}) 
 = \beta_k \beta_{k - t} y^{m + l + k - t - a_{k - t}}
\end{gather*}
have to be equal, a contradiction to that $T$ is monomial.

{\sc Step 3}.
If $n < \infty$, then applying~\eqref{Averaging} for $T(x^n)$ and $T(y^m)$,
we obtain $\alpha_0 = \alpha_n$.
So, we have the operator $T_1$ when $n < \infty$ and $T_2$ when $n = \infty$. \end{proof}

As in \S4, we have a deep connection between Rota---Baxter and averaging operators.
For $\charr F = 0$, the analogue of Question~\ref{AdditionalProblem} 
stated for operators that do not contain monomials in the kernel has a positive answer to both a) and b).

\section{Operators on $F[x, y]$ and $F_0[x, y]$ without monomials in $\Sh$}

In this section, we study monomial operators~$L$ on $F[x, y]$ and $F_0[x, y]$ under the restriction similar to the one from $\S$5: $\Sh L = \emptyset$.
This condition implies that $L(x^n y^m) = \alpha_{n, m} x^n y^m$ 
for all $n, m \geqslant 0$ ($n + m > 0$ on $F_0[x,y]$).
Therefore, the relations~\eqref{RBO} and~\eqref{Averaging} for such~$L$
convert to a system of equations on the coefficients $\alpha_{n,m}$.

\begin{theorem} \label{AvOp_Case3.1_theorem}
Let $F$ be a field of characteristic zero
and let $R$ be an RB-operator on $F_0[x, y]$ or $F[x, y]$ defined as
$R(x^n y^m) = \alpha_{n, m} x^n y^m$, $\alpha_{n, m} \in F$, $n, m \geqslant 0$.

a) On $F[x, y]$, there exists only trivial RB-operator $R = 0$.

b) On $F_0[x, y]$, either there exists $\gamma \in F$ such that
for any $s,t \geqslant 0$, $s + t > 0$, we have
\begin{gather} \label{RBO_Case3_1_F0_solution_first}
\alpha_{s, t} =
\begin{cases}
    \gamma / a, & s = al, \, t = ar, \, a > 0, \\
    0, & \text{otherwise},
\end{cases}
\end{gather}
or there exist $k_1, k_2 > 0$, $\alpha_1, \alpha_2 \neq 0$,
and some $a, b ,d\geqslant 0$ such that
$d = \gcd (k_1, k_2)$, $a < d$, $b \mid d$, $d/b \mid \gcd (a, d)$, 
and for all $s, t \geqslant 0$, $s + t > 0$, we have
\begin{gather} \label{RBO_Case3_1_F0_solution_second}
\alpha_{s, t}
 = \begin{cases}
 \dfrac{\alpha_1\alpha_2}
 {(t / k_2)\alpha_1 + (s / k_1)\alpha_2},
 & s \equiv_{k_1} (k_1/b) l,\,t \equiv_{k_2} (k_2a/d)l,\, 0 \leqslant l < b,\\
 0, & \mbox{otherwise}.
\end{cases}
\end{gather}
\end{theorem}

\begin{proof}
Substituting $a = x^n y^m$ and $b = x^s y^t$ in~\eqref{RBO}, we find
\begin{equation} \label{RBO_Case3}
\alpha_{n,m}\alpha_{s,t} = (\alpha_{n,m} + \alpha_{s,t})\alpha_{n + s, m + t},
\quad n, m, s, t \geqslant 0.
\end{equation}
Then for $\alpha_{s, t} \neq 0$ and any $n , m \geqslant 0$, the following formulas hold:
\begin{gather} 
\alpha_{n + s, m + t} = 
\frac{\alpha_{n,m}\alpha_{s,t}}{\alpha_{n,m} + \alpha_{s,t}}, \label{RBO_Case3_c=ab/(a+b)}\\
\alpha_{n + as, m + at} =
\frac{\alpha_{s, t}\alpha_{n, m}}{\alpha_{s, t} + a \alpha_{n, m}}, 
\quad a \geqslant 0, \label{RBO_Case3_(n,m)_a+1_raz_generalization}\\
\alpha_{(a + 1)s, (a + 1)t} = \alpha_{s, t} / (a + 1), 
\quad a \geqslant 0. \label{RBO_Case3_(n,m)_a+1_raz}
\end{gather}
The formula~\eqref{RBO_Case3_c=ab/(a+b)} follows from~\eqref{RBO_Case3} directly.
One may prove~\eqref{RBO_Case3_(n,m)_a+1_raz_generalization} by induction on $a$,
and~\eqref{RBO_Case3_(n,m)_a+1_raz} coincides with~\eqref{RBO_Case3_(n,m)_a+1_raz_generalization} when $\alpha_{n, m} = \alpha_{s, t}$.
If $\alpha_{n,m} + \alpha_{s,t} = 0$, then by~\eqref{RBO_Case3} we have 
$\alpha_{n,m}\alpha_{s,t} = 0$, which contradicts $\alpha_{s,t} \neq 0$.
Therefore, none of the obtained formulas involve division by zero.

{\sc Case $F[x, y]$}.\,Putting $n = m = 0$ in~\eqref{RBO_Case3}, 
we obtain $\alpha_{s, t}^2 = 0$.
So, $\alpha_{s, t} = 0$~for~all~$s, t \geqslant 0$.

{\sc Case $F_0[x, y]$}.
Suppose that there exists a nonzero $\alpha_{s, t}$.
We will consider several cases.

{\sc Case 1}: $\alpha_{0, t} = 0$ for all $t > 0$.
The case when $\alpha_{t, 0} = 0$ for all $t > 0$ is considered analogously.
Define $\emptyset \neq I\subseteq \mathbb{N}$:
$$
I = \{ i  \mid \mbox{there exists $j$ such that } \alpha_{i, j} \neq 0\}.
$$
For any  $i \in I$ there exists a~minimal $k_i$ such that $\alpha_{i, k_i} \neq 0$.

{\sc Step 1}.
We find out indexes of all nonzero coefficients and 
prove that $k_{i + j} = k_i + k_j$ for all $i, j \in I$.
Considering~\eqref{RBO_Case3_c=ab/(a+b)} with $\alpha_{s, t} \neq 0$ and $\alpha_{0, m} = 0$,
we obtain $\alpha_{s, t + m} = 0$ for any $m > 0$. Thus,
$$
\alpha_{s, t} =
\begin{cases}
    \alpha_{s, k_s}, & s\in I,\,t = k_s,\\
    0, & \mbox{otherwise}.
\end{cases}
$$

If $\alpha_{n,m} = \alpha_{s,t} = 0$, then~\eqref{RBO_Case3} is the identity $0 = 0$.
If $\alpha_{n,m}, \alpha_{s,t} \neq 0$, then~$m = k_n$ and $t = k_s$,
and by~\eqref{RBO_Case3_c=ab/(a+b)} we find that $k_{n + s} = k_n + k_s$.
If $\alpha_{n, m} \neq 0$ and $\alpha_{s, t} = 0$, 
then~\eqref{RBO_Case3} implies 
$k_n + t = m + t \neq k_{n + s} = k_n + k_s$, which holds, since~$t \neq k_s$.

{\sc Step 2}.
Let $l$ be a minimal number from $I$ such that $\alpha_{l, k_l} \neq 0$.
Now we prove that for any $n \in I$, the equality $n k_l = l k_n$ holds.
Assume that $n k_l - l k_n > 0$.
Then by~\eqref{RBO_Case3_c=ab/(a+b)} and~\eqref{RBO_Case3_(n,m)_a+1_raz} we arrive at a~contradiction:
$$
0 \neq
\frac{\alpha_{l, k_l}}{n} \stackrel{\eqref{RBO_Case3_(n,m)_a+1_raz}}{=}
\alpha_{nl, nk_l} \stackrel{\eqref{RBO_Case3_c=ab/(a+b)}}{=} 
\frac{\alpha_{nl, lk_n}\alpha_{0, n k_l - l k_n}}
{\alpha_{nl, lk_n} + \alpha_{0, n k_l - l k_n}} =
0.
$$
If~$l k_n - n k_l > 0$, then we also get a~contradiction:
$$
0 \neq
\frac{\alpha_{n, k_n}}{l} =
\alpha_{ln, lk_n} =
\frac{\alpha_{ln, nk_l}\alpha_{0,lk_n - nk_l}}
{\alpha_{ln, nk_l} + \alpha_{0,lk_n - nk_l}}.
$$

{\sc Step 3}.
We prove the formula~\eqref{RBO_Case3_1_F0_solution_first}.
Let $(a,b)\in \mathbb{N}\times \mathbb{N} \setminus \{(0,0)\}$ be a~such pair that $ak_l = l b$.
Among all such pairs we take the one $(n,m)$ with a minimal value of~$n + m$.
If $l = n$, then
$$
\alpha_{s, t} =
\begin{cases}
    \alpha_{l , k_l} / a, & s = al, \, t = ak_l, \, a > 0, \\
    0, & \text{otherwise}.
\end{cases}
$$
If $n<l$, then there exists $r > 1$ such that $l = rn$ 
and for all $1 \leqslant c < r$ we have $\alpha_{cn, cm} = 0$.
Then by~\eqref{RBO_Case3_c=ab/(a+b)} we have $\alpha_{al + cn, ak_l + cm} = 0$ 
for all $0 < a$ and $1 \leqslant c < r$.
Thus, we have proved the formula~\eqref{RBO_Case3_1_F0_solution_first}
in each of the cases considered.

{\sc Case 2}:
$\alpha_{0, t} = 0$ and $\alpha_{s, 0} = 0$ not for all $s, t > 0$. 
Hence there exist minimal $k_1$ and~$k_2$
such that $\alpha_{k_1, 0}, \alpha_{0, k_2} \neq 0$.

{\sc Step 1}.
At this step, we will determine how nonzero coefficients~$\alpha_{n,m}$ are related.
The formulas~\eqref{RBO_Case3_c=ab/(a+b)} and~\eqref{RBO_Case3_(n,m)_a+1_raz} imply
\begin{gather}\label{RBO_Case3_setka_equation}
\alpha_{i k_1, j k_2} = 
\frac{\alpha_{k_1, 0}\alpha_{0, k_2}}
{j \alpha_{k_1, 0} + i\alpha_{0, k_2}}, 
\,\, i + j > 0, \,\, i, j \geqslant 0.
\end{gather}
Define the set
$\emptyset \neq I \subseteq \mathbb{N} \times \mathbb{N}$ as follows:
\begin{equation} \label{section_Sh=empty_def_set_I}
I^2 = \{ (i, j) \mid 0 < i < k_1, \, 0 < j < k_2, \, \alpha_{i, j} \neq 0\}.
\end{equation}
Applying~\eqref{RBO_Case3_c=ab/(a+b)} and~\eqref{RBO_Case3_setka_equation} for $(u, v)$, 
where $0\leqslant u<k_1$, $0\leqslant v<k_2$, $0 < u + v$, 
we obtain
\begin{gather}\label{RBO_Case3_sootvetstvie_v_kvadratah}
\alpha_{u, v} \neq 0\ 
\Longleftrightarrow\
\alpha_{i k_1 + u, j k_2 + v} \neq 0 
\text{ for all }
i, j \geqslant 0.
\end{gather}
Thus, the RB-operator will be completely determined if we find 
all pairs of indexes from~$I^2$, 
this is enough by~\eqref{RBO_Case3_sootvetstvie_v_kvadratah}.
If $I^2 = \emptyset$, then all nonzero coefficients are determined by~\eqref{RBO_Case3_setka_equation},
this formula is a~particular case of~\eqref{RBO_Case3_1_F0_solution_second}.

{\sc Step 2}.
At this step we describe $I^2$.
Suppose that $I^2 \neq \emptyset$.
Note that if $(s, t) \in I^2$, then
we have $(s + a, t) \not \in I^2$ and $(s, t + b) \not\in I^2$ 
for all $a, b$ such that $0 < s + a < k_1$, $0 < t + b < k_2$.
Let us prove that $(s + a, t)\not \in I^2$,
the proof of the other case is similar.
Due to minimality of $k_1$, we have $\alpha_{c, 0} = 0$ for every $0 < c < k_1$.
So, by~\eqref{RBO_Case3_c=ab/(a+b)} we obtain 
$$
\alpha_{s + a, t} = 
\frac{\alpha_{s, t} \alpha_{a, 0}}{\alpha_{s, t} + \alpha_{a, 0}}, 
\quad 0 < a < k_1.
$$
It is easy to see that by~\eqref{RBO_Case3_sootvetstvie_v_kvadratah}
we have proved the statement.
In other words, we have shown that if $(s, t) \in I^2$,
then $(s, r), (l, t) \not\in I^2$ for all $0 < r < k_2$, $r \neq t$, $0 < l < k_1$, $l \neq s$.

{\sc Step 3}.
Let us simplify the description of $I^2$.
We fix $d = \gcd (k_1, k_2)$ and denote $k_1 = dk_1^\prime$, $k_2 = dk_2^\prime$.
Now we prove that $I^2$ contains only pairs $(i, j)$ such that 
$i = ak_1^\prime$ and $j = bk_2^\prime$ for some $a,b\in \mathbb{N}$.
By the choice of $k_1$ and $k_2$ and by~\eqref{RBO_Case3_sootvetstvie_v_kvadratah},
if $\alpha_{ak_1, bk_2 + q} \neq 0$, where $0 \leqslant q < k_2$, 
then $q = 0$.
Analogously, if $\alpha_{ak_1 + c, bk_2} \neq 0$, where $0 \leqslant c < k_1$, then $c = 0$.
Consider a pair $(s, t) \in I^2$ and define $a = k_1 /\gcd (k_1, s)$, 
thus $as = k_1r$ for $r = s / \gcd (k_1, s)$.
Then by~\eqref{RBO_Case3_(n,m)_a+1_raz} we conclude that $k_2 \mid at$, 
hence $dk_2^\prime \mid dtk_1^\prime / \gcd (k_1, s)$ and so $k_2^\prime \mid t$.
Similarly, we get $k_1^\prime \mid s$.
Therefore, for any $(i, j) \in I^2$ there exist $a, b$ such that $(i, j) = (ak_1^\prime, bk_2^\prime)$.

{\sc Step 4}.
Now we describe the set $I^2$ directly.
Define the following set:
$$
J^2 = \{(a, b) \mid (ak_1^\prime, bk_2^\prime) \in I^2\}.
$$
The set $J^2\cup \{(0,0)\}$ determines a subgroup in $\mathbb{Z}_d \times \mathbb{Z}_d$.
Moreover, we get the following subgroups in $\mathbb{Z}_d$:
$$
A = \{s \mid (s, t) \in J^2 \}\cup \{0\}, \quad
B = \{t \mid (s, t) \in J^2 \}\cup \{0\}.
$$
By the Lagrange's theorem, $d_1, d_2\,|\,d$, where $d_1 = |A|$ and $d_2 = |B|$.

We have $J^2 \neq \emptyset$, since $I^2 \neq \emptyset$,
it implies that there exists $(d/d_1, a) \in J^2$, $0 \leqslant a < d$.
Hence
$$
J^2 = \{(sd/d_1, sa \!\!\! \mod \, d) \mid s = 1, \ldots, d_1 -1\}.
$$
Due to the definition of $J^2$, if $(s, t) \in J^2$, 
then $(sk_1^\prime, tk_2^\prime) \in I^2$, i.\,e.
$\alpha_{sk_1^\prime, tk_2^\prime} \neq 0$.
Let us show that $\gcd (a, d) = ed/d_1$ for some $e \geqslant 1$. 
Indeed, if $d/d_1$ does not divide $a$, then we arrive at a contradiction:
$$
0 \neq
\alpha_{d k_1^\prime/ d_1, a k_2^\prime} / d_1 \stackrel{\eqref{RBO_Case3_(n,m)_a+1_raz}}{=} 
\alpha_{d_1 k_1^\prime d / d_1, d_1 a k_2^\prime } = 
\alpha_{k_1^\prime d, d_1 a k_2^\prime} = 
\alpha_{k_1, d_1 a k_2^\prime} = 
0.
$$
The last equality holds, since 
$d_1 a k_2^\prime / dk_2^\prime = d_1 a / d \not \in \mathbb{Z}$
and so $d_1 a k_2^\prime \not\equiv_{dk_2^\prime} 0$.

{\sc Step 5}.
Finally, we prove the formula~\eqref{RBO_Case3_1_F0_solution_second}.
By~\eqref{RBO_Case3_(n,m)_a+1_raz}, we see that
$\alpha_{s k_1/d_1, s a k_2^\prime} = \alpha_{k_1/d_1, a k_2^\prime} / s$,  $0 < s < d_1$.
Further, we express the coefficient $\alpha_{k_1/d_1, a k_2^\prime}$
via $\alpha_{k_1, 0}, \alpha_{0, k_2}$ from the following relation:
$$
\frac{\alpha_{k_1, 0} \alpha_{0, k_2}}
{(a d_1/d)\alpha_{k_1, 0} + \alpha_{0, k_2}} \stackrel{\eqref{RBO_Case3_c=ab/(a+b)}}{=} 
\alpha_{k_1, a d_1 k_2 / d} = 
\alpha_{k_1, d_1 a_1 k_2^\prime} \stackrel{\eqref{RBO_Case3_(n,m)_a+1_raz}}{=} 
\alpha_{k_1/d_1, a k_2^\prime}/d_1, 
$$
so we find that 
$
\alpha_{k_1/d_1, a k_2^\prime} = \frac{d\alpha_{k_1, 0} \alpha_{0, k_2}}
{a\alpha_{k_1, 0} + (d/ d_1)\alpha_{0, k_2}}.
$
Applying this equality and~\eqref{RBO_Case3_setka_equation}, we obtain
\begin{multline*}
\alpha_{k_1 i + k_1/d_1 s, k_2 j + a k_2^\prime s} \stackrel{\eqref{RBO_Case3_c=ab/(a+b)}}{=}  
\alpha_{k_1/d_1 s, a k_2^\prime s}\alpha_{k_1i, k_2j} / 
(\alpha_{k_1/d_1 s, a k_2^\prime s} + \alpha_{k_1i, k_2j}) \\ = 
\frac{d\alpha_{k_1, 0} \alpha_{0, k_2}}{(a\alpha_{k_1,0} + d/d_1\alpha_{0, k_2})s} \cdot
\frac{ \alpha_{k_1, 0} \alpha_{0, k_2}}{j\alpha_{k_1, 0} + i\alpha_{0, k_2}} / 
\left(
\frac{d\alpha_{k_1, 0} \alpha_{0, k_2}}
{(a\alpha_{k_1, 0} + d/d_1\alpha_{0, k_2})s} 
+ 
\frac{ \alpha_{k_1, 0} \alpha_{0, k_2}}
{ j\alpha_{k_1, 0} + i\alpha_{0, k_2}}
\right) \\ =
\frac{\alpha_{k_1,0}\alpha_{0, k_2}}
{(j + a s / d)\alpha_{k_1,0} + (i + s / d_1)\alpha_{0, k_2}}.
\end{multline*}

Substituting the coefficients $\alpha_{n,m}$, $\alpha_{s,t}$ 
expressed by~\eqref{RBO_Case3_1_F0_solution_second} in~\eqref{RBO_Case3}, 
we obtain a trivial identity.
\end{proof}

\begin{remark}
Denote $z = x^{k_1/d_1}$ and $t = y^{k_2a/d}$.
Then the formula~\eqref{RBO_Case3_1_F0_solution_second} actually coincides
with the one from Example~5~\cite{MonomNonunital} devoted 
to the injective RB-operator of weight~0~on~$F_0[z,t]$.
\end{remark}

\begin{remark}
Let $L$ be a linear operator defined by formula~\eqref{RBO_Case3_1_F0_solution_second}.
It is an RB-operator provided that $\alpha_1 / \alpha_2 \not \in \mathbb{Q}_{<0}$.
\end{remark}

\begin{example}
Suppose that in Theorem~\ref{AvOp_Case3.1_theorem} $\alpha_{n, m} \neq 0$ for all $n, m \geqslant 0$.
Then the formula~\eqref{RBO_Case3_1_F0_solution_second} takes the following form:
$$
\alpha_{s, t} = \frac{\alpha_1\alpha_2}{t\alpha_1 + s\alpha_2}, \,\, s, t \geqslant 0, \,\, s + t > 0.
$$
\end{example}

\begin{example} \label{example_for_picture}
Take $k_1 = 16$, $k_2 = 12$, $a = b = 2$, $d = 4$.
So,~\eqref{RBO_Case3_1_F0_solution_second} has the following form:
$$
\alpha_{16 i + p, 12 j + q} = 
\begin{cases}
\dfrac{\alpha_1\alpha_2}
{(j + \frac{k}{2})\alpha_1 + (i + \frac{k}{2})\alpha_2}, 
& p = 8k,\, q = 6k,\, k \in \{0, 1\},\, i, j \geqslant 0, \\
0, & \mbox{otherwise}.
\end{cases}
$$
\end{example}

\begin{example}
Take $k_1 = 32$, $k_2 = 48$, $a = 12$, $b = 4$, $d = 16$.
Then~\eqref{RBO_Case3_1_F0_solution_second} may be rewritten as follows:
$$
\alpha_{s, t} = 
\begin{cases}
\dfrac{\alpha_1\alpha_2}
{\frac{t}{48}\alpha_1 + \frac{s}{32}\alpha_2},
& s \equiv_{32} 8k,\, t \equiv_{48} = 36k,\, 0 \leqslant k \leqslant 3, \\
0, & \mbox{otherwise}.
\end{cases}
$$
\end{example}

\begin{center}
\pgfplotsset
{
width = 14cm, 
height = 8cm, 
compat=1.15
}
\begin{tikzpicture}
\tikzstyle{every node}=[font=\small]
\begin{axis}
[
    axis x line = middle,
    axis y line = left,
    xtick={0,8,16,24,32},
    ytick={0,6,12,18,24,30,36},
    xmin = 0,
    ymin = 0,
    xmax = 36,
    ymax = 40,
    grid = major
]

\addplot[fill=ForestGreen!10!white,draw=none] 
coordinates {
(0, 0)
(0, 12)
(16, 12)
(16, 0)
};

\addplot
[
    mark=*,
    blue!65!white,
    mark size=3pt,
    only marks,
    point meta=explicit symbolic
]
coordinates 
{
    (0, 12)
    (0, 24)
    (0, 36)
    (8, 6)
    (8, 18)
    (8, 30)
    (16, 0)
    (16, 12)
    (16, 24)
    (16, 36)
    (24, 6)
    (24, 18)
    (24, 30)
    (32, 0)
    (32, 12)
    (32, 24)
    (32, 36)
};

\addplot [thick, lightgray]
coordinates 
{
    (8, -10)
    (8, 40)
};
\addplot [thick, lightgray]
coordinates 
{
    (24, -10)
    (24, 40)
};
\addplot [thick, lightgray]
coordinates 
{
    (0, 6)
    (36, 6)
};
\addplot [thick, lightgray]
coordinates 
{
    (0, 18)
    (36, 18)
};
\addplot [thick, lightgray]
coordinates 
{
    (0, 30)
    (36, 30)
};
\addplot [thick, black]
coordinates 
{
    (16, -10)
    (16, 50)
};
\addplot [very thick, black]
coordinates 
{
    (0, -10)
    (0, 39.45)
};
\addplot [very thick, black]
coordinates 
{
    (-10, 0)
    (35.75, 0)
};
\addplot [thick, black]
coordinates 
{
    (32, -10)
    (32, 50)
};
\addplot [thick, black]
coordinates 
{
    (0, 12)
    (50, 12)
};
\addplot [thick, black]
coordinates 
{
    (0, 24)
    (50, 24)
};
\addplot [thick, black]
coordinates 
{
    (0, 36)
    (50, 36)
};

\node [left, blue!65!white] at (3.7, 14) {$(0, 12)$};
\node [left, blue!65!white] at (11.2, 8) {$(8, 6)$};
\node [left, blue!65!white] at (11.7, 20) {$(8, 18)$};
\node [left, blue!65!white] at (19.7, 2) {$(16, 0)$};
\node [left, blue!65!white] at (20.3, 14) {$(16, 12)$};
\node [left, blue!65!white] at (27.7, 8) {$(24, 6)$};
\node [left, blue!65!white] at (28.3, 20) {$(24, 18)$};
\end{axis}
\end{tikzpicture}

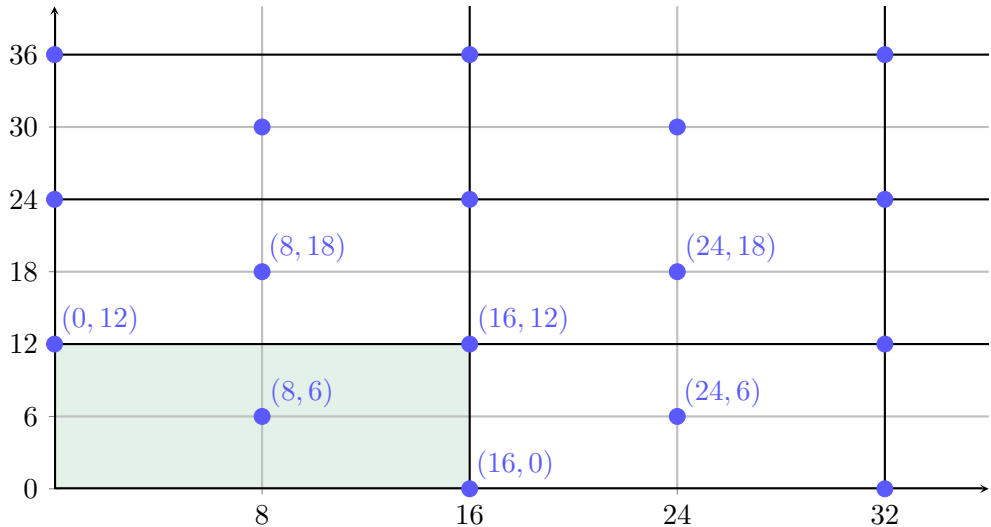
Fig.~\arabic{pict}. Geometric interpretation of Example~\ref{example_for_picture}.
\end{center}

Let us consider a geometric interpretation of the action of the RB-operator~$R$ 
from Example~\ref{example_for_picture} on monomials (see Fig.~\arabic{pict}).
There is a one-to-one correspondence between points $(s,t)$ with $s,t\in \mathbb{N}$, $s+t>0$, and monomials $x^s y^t$. 
If $\alpha_{s, t} \neq 0$ for some point $(s,t)$, then we color this point blue.
Moreover, we have a grid defined by points $(s,t)$ such that $k_1\,|\,s$
and $k_2\,|\,t$. We mark the grid on Fig.~1 with black.
Parameters $a, b$, and $d$ determine points $(s,t)$ lying inside rectangles of the grid and satisfying $\alpha_{s,t}\neq 0$.
By the proof of Theorem~\ref{AvOp_Case3.1_theorem}, 
the location of the blue points in each of the grid rectangles is completely determined by that in the lower left rectangle, marked in pale green.

Now we consider an analogue of Theorem~\ref{AvOp_Case3.1_theorem} for averaging operators.

\begin{theorem} \label{AvOp_Case3.1_theorem_averaging}
Let $T$ be an averaging operator on $F_0[x, y]$  or $F[x, y]$
defined as  $T(x^n y^m) = \alpha_{n, m} x^n y^m$, $\alpha_{n, m} \in F$, $n, m \geqslant 0$.

a) On $F[x, y]$, either there exists $\gamma \in F$ such that
for any $s, t \geqslant 0$ we have
\begin{gather} \label{Averaging_Case3_1_F0_solution_first}
\alpha_{s, t} =
\begin{cases}
    \gamma, & s = al, \, t = ar, \, a \geqslant 0, \\
    0, & \text{otherwise},
\end{cases}
\end{gather}
or there exist $k_1, k_2 > 0$, $\gamma \neq 0$, 
and some $a, b, d \geqslant 0$ such that 
$d = \gcd (k_1, k_2)$, $ a < d$, $b \mid d$,  $d/b \mid \gcd (a, d)$, 
and for all $s, t \geqslant 0$ we have
\begin{gather} \label{Averaging_Case3_1_F0_solution_second}
\alpha_{s, t} = 
\begin{cases}
\gamma, & s \equiv_{k_1} (k_1/b) l,\,t \equiv_{k_2} (k_2a/d)l,\, 0 \leqslant l < b,\\
0, & \mbox{otherwise}.
\end{cases}
\end{gather}

b) On $F_0[x, y]$, we have the same two families of coefficients $\alpha_{s,t}$ as in a) with restriction $s + t > 0$ in both 
formulas~\eqref{Averaging_Case3_1_F0_solution_first},~\eqref{Averaging_Case3_1_F0_solution_second},
and also $a > 0$ in~\eqref{Averaging_Case3_1_F0_solution_first}.
\end{theorem}

\begin{proof}
Like in the proof of Theorem~\ref{Averaging_not_incr_and_ker=0},
we will follow the steps from the proof of Theorem~\ref{AvOp_Case3.1_theorem}
avoiding technical details.
We consider both algebras $F_0[x, y]$ and $F[x, y]$ simultaneously.

Substituting $a = x^n y^m$ and $b = x^s y^t$ in~\eqref{Averaging}, we obtain the following recurrence relations:
\begin{equation} \label{Averaging_case3_analogue_eq1}
\alpha_{n,m}\alpha_{s,t} = 
\alpha_{n,m}\alpha_{n + s, m + t} =
\alpha_{s,t}\alpha_{n + s, m + t},
\quad n, m, s, t \geqslant 0.
\end{equation}
Let $\alpha_{n,m}$ and $\alpha_{s,t}$ be any nonzero coefficients, 
then we derive by~\eqref{Averaging_case3_analogue_eq1} that $\alpha_{n,m} = \alpha_{s,t}$.
So, all nonzero coefficients are equal.
Denote by $\gamma$ the value of all nonzero coefficients, i.\,e.
if $\alpha_{s,t}\neq0$, then $\alpha_{s, t} = \gamma$.

Applying~\eqref{Averaging_case3_analogue_eq1} for $\alpha_{s, t} \neq 0$,
we get the formulas similar 
to~\eqref{RBO_Case3_(n,m)_a+1_raz_generalization} and~\eqref{RBO_Case3_(n,m)_a+1_raz}:
\begin{gather} \label{Averaging_case3_analogue_eq2}
\alpha_{n + as, m + at} = \alpha_{n, m}, \quad 
\alpha_{(a + 1)s, (a + 1)t} = \alpha_{s, t}, \quad 
a \geqslant 0. 
\end{gather}

Unlike the proof of Theorem~\ref{AvOp_Case3.1_theorem}, 
we need to take into account the coefficient $\alpha_{0, 0}$ 
in the case of the $F[x, y]$.
All other arguments will remain unchanged for both algebras $F_0[x,y]$ and $F[x,y]$.
Substituting $n = m = 0$ in~\eqref{Averaging_case3_analogue_eq1},
we obtain $\alpha_{0, 0} \alpha_{s, t} = \alpha_{s, t}^2$.
If $\alpha_{0, 0} = 0$, then we get $T = 0$.
If $\alpha_{0, 0} \neq 0$, then $\alpha_{0, 0} = \gamma$ and both 
formulas~\eqref{Averaging_Case3_1_F0_solution_first},~\eqref{Averaging_Case3_1_F0_solution_second} include this case.

Suppose that there exists a nonzero $\alpha_{s, t}$.

{\sc Case 1}: $\alpha_{0, t} = 0$ for all $t > 0$.
Now we define a nonempty set $I\subseteq \mathbb{N}$:
$$
I = \{ i  \mid \mbox{there exists $j$ such that } \alpha_{i, j} \neq 0\}.
$$
For any $i \in I$, there exists a minimal $k_i$ such that $\alpha_{i, k_i} \neq 0$.

{\sc Step 1}.
Due to~\eqref{Averaging_case3_analogue_eq1} 
we conclude that $k_{s + t} = k_s + k_t$ for all $s, t \in I$.

{\sc Step 2}.
Consider a~minimal $l \in I$ such that $\alpha_{l, k_l} \neq 0$.
Let us prove that $nk_l = lk_n$ holds for any $n \in I$.
Assume to the contrary that $nk_l \neq lk_n$.
First, we consider $nk_l > lk_n$ and arrive at a contradiction:
$$
0 \neq
\alpha_{l, k_l} 
\stackrel{\eqref{Averaging_case3_analogue_eq2}}{=} 
\alpha_{nl, nk_l} 
\stackrel{\eqref{Averaging_case3_analogue_eq1}}{=} 
\frac{\alpha_{nl, lk_n}\alpha_{0, nk_l - lk_n}}{\alpha_{nl, lk_n}} = 
0.
$$
Second, we consider the case $nk_l <lk_n$ and again arrive at a contradiction:
$$
0 \neq
\alpha_{n, k_n} =
\alpha_{ln, l k_n} =
\frac{\alpha_{ln, nk_l}\alpha_{0, lk_n - nk_l}}{\alpha_{ln, nk_l}} = 
0.
$$

{\sc Step 3}.
We consider pairs $(a,b)\in \mathbb{N}\times \mathbb{N} \setminus \{(0,0)\}$ 
such that $ak_l = l b$.
Similarly to the proof of Theorem~\ref{AvOp_Case3.1_theorem},
formula~\eqref{Averaging_Case3_1_F0_solution_first} can be derived
from~\eqref{Averaging_case3_analogue_eq1} and~\eqref{Averaging_case3_analogue_eq2}. 

{\sc Case 2}: $\alpha_{0, t} = 0$ and $\alpha_{s, 0} = 0$
not for all $s, t > 0$. Hence there exist minimal $k_1$ and~$k_2$
such that $\alpha_{k_1, 0}, \alpha_{0, k_2} \neq 0$.
As noted in the beginning of the proof, $\alpha_{k_1, 0} = \alpha_{0, k_2} = \gamma$.

{\sc Step 1}.
As in the proof of Theorem~\ref{AvOp_Case3.1_theorem}, 
we have the formula
\begin{gather}\label{Averaqing_Case3_setka_equation}
\alpha_{i k_1, j k_2} =  \gamma,
\,\, i + j > 0, \,\, i, j \geqslant 0.
\end{gather}
We also introduce the set~$I$ defined by the formula~\eqref{section_Sh=empty_def_set_I}.
Applying~\eqref{Averaging_case3_analogue_eq2} and~\eqref{Averaqing_Case3_setka_equation} for $(u, v)$, 
where $0\leqslant u<k_1$, $0\leqslant v<k_2$, $0 < u + v$, 
we obtain
\begin{gather}\label{Averaging_Case3_sootvetstvie_v_kvadratah}
\alpha_{u, v} \neq 0\ 
\Longleftrightarrow\
\alpha_{i k_1 + u, j k_2 + v} \neq 0 
\text{ for all }
i, j \geqslant 0.
\end{gather}
In order to prove the formula~\eqref{Averaging_case3_analogue_eq1},
it is again enough to determine the set~$I^2$.
If $I^2 = \emptyset$, then we describe all nonzero coefficients 
by~\eqref{Averaqing_Case3_setka_equation},
which is a particular case of~\eqref{Averaging_Case3_1_F0_solution_second}.

{\sc Step 2}.
Suppose that $I^2 \neq \emptyset$.
As in the case of RB-operators, we may prove that if $(s, t) \in I^2$, 
then $(s + a, t) \not \in I^2$ and $(s, t + b) \not\in I^2$
for all $a, b$ such that $0 < s + a < k_1$, $0 < t + b < k_2$.

{\sc Step 3}.
We repeat the arguments of the corresponding step from the proof of Theorem~\ref{AvOp_Case3.1_theorem}.
We also define a~parameter $d = \gcd (k_1, k_2)$ and show that for any pair $(i, j) \in I^2$ we have $(i, j) = (a k^\prime_1, b k^\prime_2)$,
where $a, b \geqslant 0$ and $k_i = dk^\prime_i$, $i = 1, 2$.

{\sc Step 4}.
We define the set $J^2$ as subgroup of $\mathbb{Z}_d \times \mathbb{Z}_d$.
Further, we define $d_1$ and $a$ as numbers such that 
$d_1 \mid d$ and $d/d_1 \mid \gcd (a, d)$.

{\sc Step 5}.
We have already proved that all nonzero coefficients are pairwise equal,
so the relations~\eqref{Averaging_case3_analogue_eq1} hold.
Hence, the formula~\eqref{Averaging_Case3_1_F0_solution_second} has been proved.
\end{proof}

For $\charr F = 0$, the analogue of Question~\ref{AdditionalProblem}
stated for averaging operators and RB-operators on $F_0[x, y]$ satisfying the conditions of 
Theorems~\ref{AvOp_Case3.1_theorem} and~\ref{AvOp_Case3.1_theorem_averaging}
respectively, has a~positive answer for both a) and b).
In particular, we can construct all monomial RB-operators satisfying the condition 
$\Sh = \emptyset$ by corresponding monomial averaging operators.
On $F[x, y]$, the analogue of Question~\ref{AdditionalProblem} 
for these sets of operators is positive for a) and negative for~b),
since there exists only zero such RB-operator on $F[x,y]$.

\section{Conclusion}

We have seen in sections $\S$4--6 that there is a deep connection between monomial averaging operators and RB-operators.
In case of polynomials in one variable, we can construct any RB-operator
by some averaging operator.
However, the situation changes when we consider polynomials with two variables.
Although the sets of all monomial averaging operators and RB-operators on $F_0[x,y]$ ($F[x,y]$) 
have more significant differences than on $F_0[x]$ ($F[x]$), 
their subsets determined by some natural conditions may (in some sense) coincide.
As an example of such a restriction, we consider non-increasing in degree operators
whose kernel does not contain monomials.
However, as we can see in the case of operators satisfying the condition $\Sh = \emptyset$, the set of monomial averaging operators is in some sense larger than the set of monomial Rota---Baxter ones. 
Anyway, we suggest the following method of the classification of RB-operators on a given algebra: it may be very useful to start the problem with description of averaging operators. In addition, averaging operators satisfy a larger number of relations, and this simplifies their description.
An open question is to find the connection between Rota---Baxter operators of weight zero and averaging ones outside the context of monomiality.

\section*{Acknowledgments}

Author is grateful to his supervisor Vsevolod Gubarev for very valuable discussions
and to Maxim Goncharov, who read the first draft of this paper and pointed out several errors.

The study was supported by a grant from the Russian Science Foundation \linebreak
\textnumero 23-71-10005, https://rscf.ru/project/23-71-10005/


\begin{thebibliography}{67}
\bibitem{GuoMonograph}
L.~Guo,
An Introduction to Rota--Baxter Algebra. 
Surveys of Modern Mathematics, vol. 4, Intern. Press, Somerville (MA, USA); Higher Education Press, Beijing, 2012.



\bibitem{Baxter}
G.~Baxter, 
An analytic problem whose solution follows from a simple algebraic identity,
Pacific J. Math. \textbf{10} (1960), 731--742.

\bibitem{Tricomi}
F.G.~Tricomi,
On the finite Hilbert transform,
Quart. J. Math. \textbf{2} (1951), 199--211.

\bibitem{Cotlar}
M.~Cotlar, 
A unified theory of Hilbert transforms and ergodic theorems, 
Revista Matematica Cuyana \textbf{1} (1955), 105--167.

\bibitem{Rota}
G.-C.~Rota,
Baxter algebras and combinatorial identities. I,
Bull. Amer. Math. Soc. {\bf 75} (1969), 325--329.

\bibitem{BelaDrin82}
A.A.~Belavin, V.G.~Drinfel'd.
Solutions of the classical Yang---Baxter equation for simple Lie algebras,
Funct. Anal. Appl. (3) {\bf 16} (1982), 159--180.

\bibitem{Semenov83}
M.A.~Semenov-Tyan-Shanskii,
What is a classical $r$-matrix?
Funct. Anal. Appl. {\bf 17} (1983), 259--272.

\bibitem{Guo2020}
L.~Guo, H.~Lang, Yu.~Sheng,
Integration and geometrization of Rota---Baxter Lie algebras,
Adv. Math. {\bf 387}, 107834 (2021).
\bibitem{Goncharov2020}

M.~Goncharov,
Rota---Baxter operators on cocommutative Hopf algebras.
J. Algebra {\bf 582}, 39--56 (2021).

\bibitem{Coalgebras}
H.~Zheng, C.~Zhao, L.~Liu, 
Rota---Baxter co-operators on commutative Hopf algebras. 
Communications in Algebra, (8) {\bf 52}, 3583--3595 (2024). 

\bibitem{Reynolds}
O.~Reynolds, 
On the Dynamic Theory of Incompressible Viscous
Fluids, Phil. Trans. Roy. Soc. A136 (1895), 123--164.

\bibitem{KampeDeFeriet}
J.~Kamp\'{e}~de~F\'{e}riet, 
Transformations de Reynolds op\'{e}rant dans un
ensemble de fonctions mesurables non n\'{e}gatives, C. R. Acad. Sci.
Paris {\bf 239} (1954), 787--789.

\bibitem{RotaAvOp}
G.~Rota, 
Reynolds Operators, Proceedings of Symposia in Applied
Mathematics, Vol. XVI (1964), Amer. Math. Soc., Providence,
R.I., 70--83.

\bibitem{Moy}
S-T.C.~Moy, 
Characterizations of Conditional Expectation as a Transformation on Function Spaces, 
Pacific J. Math. {\bf 4} (1954), 47--63.

\bibitem{Bong}
N.H.~Bong, 
Some Apparent Connection Between Baxter and Averaging Operators, 
J.~Math. Anal. Appl. (2) {\bf 56} (1976), 330--345.

\bibitem{PeiGuoAveraging}
J.~Pei, L.~Guo, 
Averaging algebras, Schröder numbers, rooted trees and operads, 
J.~Algebr. Comb. {\bf 42} (2015), 73--109. 

\bibitem{Monom2}
S.H.~Zheng, L.~Guo, M.~Rosenkranz,
Rota--Baxter operators on the polynomial algebras, integration and averaging operators, Pacific J. Math. (2) {\bf 275} (2015), 481--507.

\bibitem{Monom}
H.~Yu,
Classification of monomial Rota--Baxter operators on $k[x]$,
J.~Algebra Appl. {\bf 15} (2016), 1650087,~16~p.

\bibitem{MonomNonunital}
V.~Gubarev,
Monomial Rota--Baxter operators on free commutative non-unital algebra,
Siberian Electron. Math. Rep. {\bf 17} (2020), 1052--1063.

\bibitem{ReprRB}
L.~Qiao, J.~Pei,
Representations of polynomial Rota--Baxter algebras,
J. Pure Appl. Algebra (7) {\bf 222} (2018), 1738--1757.

\bibitem{ReprRB2}
X.~Tang,
Modules of polynomial Rota-Baxter algebras and matrix equations (2020), 
arXiv:2003.05630,~16~p.

\bibitem{ReprRB3}
X.~Tang, N.~Liu,
Modules of non-unital polynomial Rota-Baxter algebras,
Algebras Represent. Theory {\bf 26} (2023), 1295–-1318.

\bibitem{GubPer}
V.~Gubarev, A.~Perepechko, 
Injective Rota-Baxter operators of weight zero on $F[x]$,
Mediterr. J. Math. (6) {\bf 18} (2021), N267.

\bibitem{Ogievetsky}
O.~Ogievetsky, T.~Popov,
$R$-matrices in rime,
Adv. Theor. Math. Phys. {\bf 14} (2010), 439--506.

\bibitem{Viellard-Baron}
E.~Viellard-Baron, 
\'{E}calle’s averages, Rota–Baxter algebras and the construction of moulds (2019), 
arXiv:1904.02417v1.

\bibitem{Khodzitskii}
A.~Khodzitskii, 
Monomial Rota–Baxter Operators of Nonzero Weight on $F[x, y]$ Coming from Averaging Operators, 
Mediterr. J. Math. {\bf 20} (2023), No~251.



\bibitem{Essen}
A.~van~den~Essen, X.~Sun, 
Monomial preserving derivations and Mathieu-Zhao subspaces, 
J. Pure Appl. Algebra (10) {\bf 222} (2018), 3219--3223.

\bibitem{Kitazawa}
C.~Kitazawa, H.~Kojima, T.~Nagamine,
Closed polynomials and their applications for computations of kernels of monomial derivations, 
J. Algebra {\bf 533} (2019), 266--282.

\bibitem{NowZiel}
A.~Nowicki, J.~Zieli\'{n}ski, 
Rational constants of monomial derivations, 
J. Algebra (1) {\bf 302} (2006), 387--418.

\bibitem{Ollagnier}
J.M.~Ollagnier, A.~Nowicki, 
Monomial derivations, 
Comm. Algebra (9) {\bf 39} (2011), 3138--3150.

\bibitem{Cao}
W.~Cao,
An Algebraic Study of Averaging Operators (2014).
arXiv: 1401.7389,~75~p.


\bibitem{BGP}
P.~Benito, V.~Gubarev, A.~Pozhidaev,
Rota---Baxter operators on quadratic algebras,
Mediterr. J. Math. {\bf 15} (2018), 23~p. (N189).

\bibitem{Burde}
D.~Burde, 
Left-symmetric algebras, or pre-Lie algebras in geometry and physics, 
Cent. Eur. J. Math. {\bf 4} (2006), 323--357.

\end{thebibliography}
\end{document}